\documentclass[a4paper,12pt,twoside]{article}
\usepackage{wiaspreprint}
\usepackage[utf8]{inputenc}
\usepackage[T1]{fontenc}
\usepackage[english]{babel}
\usepackage{amsmath,amssymb,amsfonts,amsthm,mathptmx}

\usepackage{hyperref}
\usepackage{color,nicefrac}

\theoremstyle{remark}
\newtheorem{remark}{Remark}[section]
\theoremstyle{definition}
\newtheorem{theorem}{Theorem}[section]
\newtheorem{definition}[theorem]{Definition}

\newtheorem{proposition}[theorem]{Proposition}
\newtheorem{lemma}[theorem]{Lemma}

\DeclareMathOperator{\R}{\mathbb{R}}

\DeclareMathOperator{\C}{\mathcal{C}}

\DeclareMathOperator{\N}{\mathbb{N}}

\DeclareMathOperator{\ra}{\rightarrow}
\DeclareMathOperator{\de}{\text{d}}

\newcommand{\sym}{\text{sym}}

\newcommand{\f}[1]{{\pmb{ #1}}}
\DeclareMathOperator{\di}{\nabla\cdot}

\newcommand{\tu}{\tilde{\f u}}

\newcommand{\ov}[1]{\overline{{#1}}}

\renewcommand{\t}{\partial_t}

\DeclareMathOperator{\argmin}{arg\,min}

\newcommand{\vv}{\tilde{\f v}}

\DeclareMathOperator{\V}{\f H^1_{0,\sigma}}

\DeclareMathOperator{\Ha}{\f L^2_{\sigma}}

\DeclareMathOperator{\spa}{span}

\title{Maximal dissipative solutions for incompressible fluid dynamics}
\date{December 20, 2019 (revision: September 10, 2020)}

\begin{document}
\pagefootright{Berlin, December 20, 2019/rev. September 10, 2020}

\author[R. Lasarzik]{
Robert Lasarzik\nofnmark\footnote{Weierstrass Institute \\
Mohrenstr. 39 \\ 10117 Berlin \\ Germany \\
E-Mail: robert.lasarzik@wias-berlin.de}}
\nopreprint{2666}	
\nopreyear{2019}	
\selectlanguage{english}		
\subjclass[2010]{35D99, 35Q30, 35Q31, 76D05, 76N10}	
\keywords{existence,
uniqueness,  {Navier--Stokes}, {E}uler, incompressible, fluid dynamics, dissipative solutions%
}
\maketitle
\begin{abstract}
We introduce the new concept of maximal dissipative solutions 
for a  general class of  isothermal GENERIC systems. Under certain assumption, we show that maximal dissipative solutions are well posed as long as the  bigger class of dissipative solutions is non-empty. Applying this result to the Navier--Stokes and Euler equations, we infer global well-posedness of maximal dissipative solutions for these systems. 
The concept of maximal dissipative solutions coincides with the concept of weak solutions as long as the weak solutions inherits enough regularity to be unique. 
\end{abstract}

\tableofcontents
\section{Introduction}
Nonlinear partial differential equations  require generalized solution concepts, mainly because smooth solutions do not exist in general (see~\cite[Sec.~11.3.2]{evans}).

 Leray introduced in his seminal work~\cite{leray}, the concept of weak solutions to the Navier--Stokes equations, which is nowadays widely accepted and used for numerous different problems. 
 Often, they still lac uniqueness due to insufficient regularity properties. 
In two spatial dimension, the weak solutions are known to be unique.  
For higher space dimensions, this is not known. 
Probably the most well-known uniqueness  result is due to Serrin~\cite{serrin} (see Remark~\ref{weakvs}).

 Beside weak solutions, there is a plethora of different solutions concepts for different problems. They range from measure-valued, statistical, over viscosity to different dissipative solution concepts. 
These solution concepts have different properties, advantages, and disadvantages, but so far do not allow to show existence and uniqueness for the Navier--Stokes and Euler equations. 
The overall goal may be formulated as finding a solution concept that generalizes classical solutions and complies with Hadamard's definition of well-posedness. This states that a solution to a differential equation should  exist, be unique, and depend continuously on the given data. 
With the  article at hand, we want to propose a step in this direction. 
We follow the line of our previous work on dissipative solutions~\cite{diss} and define the concept of maximal dissipative solutions.  As we will show, maximal dissipative solutions can be shown to exist in any space dimension and be unique by construction. Additionally, we show that the solution depends continuously (in certain typologies) on the given initial value and right-hand side.
Thus, this article gives an affirmative answer on the well-posedness of the Navier--Stokes and Euler equations in any space dimension in the sense of maximal dissipative solutions. In contrast, this is not a positive answer to the well-known Millenium problem~\cite{mill} since it does not deal with weak solutions.

The idea behind a dissipative solution is that the equations do not have to be fulfilled in some distributional sense anymore, but the distance of the solution to smooth test functions fulfilling the equation only approximately  is measured in terms of the relative energy and relative dissipation (to be made precise later on). 
The concept of dissipative solutions was first introduced by Pi\`erre-Louis Lions in the context of the Euler equations~\cite[Sec.~4.4]{lionsfluid} with ideas originating from singular limits in the Boltzmann equation~\cite{LionsBoltzman}. 
It is also applied in the context of incompressible viscous
electro-magneto-hydro\-dy\-na\-mics~\cite{raymond}
 and equations of viscoelastic diffusion in polymers~\cite{viscoelsticdiff}.
For the more involved Ericksen--Leslie system or nematic electrolytes, it was found that the dissipative solution concept, in comparison to measure-valued solutions, captures the quantity of interest (see~\cite{masswertig} and~\cite{diss}) and  is also more amenable from the point of view of a Galerkin~\cite{approx} or finite-element approximation~\cite{nematicelectro}.

Since this concept proved worthwhile for more difficult systems, it may also be  a good solution concept for simpler systems such as the Navier--Stokes equations. 
A problem arises, since dissipative solutions are not unique, even though they enjoy the weak-strong uniqueness property: They coincide with a local strong solution, as long as the latter exists. 
Thus, naturally the question arises, whether it is possible to design an additional criterion in order to choose a special solution from these many different dissipative solutions in order to may gain uniqueness of the solution.

We propose a step into this direction by introducing the concept of maximal dissipative solutions.  Following Dafermos~\cite{dafermosscalar}, we want to choose the solution dissipating the most energy.  
Therefore, we select the dissipative solution that minimizes the energy.  Similar ideas are also used in~\cite{maximaldiss} or~\cite{maxbreit}.

As in the dissipative solution framework, maximal dissipative solutions are not known to fulfill the equation in distributional sense. 
But since all equations are modeled starting from energies and dissipation mechanisms, clinging to the equation may not simplify the analysis. 
Additionally, recent approaches showed that weak solution may not be physically relevant, if they exceed certain regularity assumption. 
For a given energy profile, it is known that there exist infinitely many weak solutions to the Euler equations~\cite{Isett} and to the Navier--Stokes equations~\cite{buckmaster}. Therefore, these solution concepts may not be the appropriate ones.
 Thus the time seems to be ready to consider alternative solution concepts. 
 One key idea for the proposed solution concept is that the solutions are compared via the relative energy to test functions with enough regularity to be physically meaningful as a solution, \textit{i.e.,}  exhibit no non-physical non-uniqueness.
The maximal dissipative solutions  only coincide with weak solutions, as long as the weak solution is unique.


In his seminal paper, Leray~\cite{leray} observed that a physically relevant solution to the Navier--Stokes equation only needs the energy and the dissipation to be bounded. The disadvantage of the concept of weak solution is that this does not suffice for the weak sequential compactness of the formulation. In contrast, this is  the case for the proposed concept of maximal dissipative solutions, \textit{i.e.}, it is weak sequentially stable with respect to the weak compactness properties read of the energy inequality. 
The solution concept of maximal dissipative solution has the additional advantage, that it is written as the minimizer of a convex functional. This allows to use standard methods from the calculus of variations for the existence proof~(see the proof of Theorem~\ref{thm:gen} below) and minimizers of functionals often exhibit additional regularity, or are more amenable for regularity estimates~(see~\cite{gil}). As it is the case for the Ericksen--Leslie equations, we hope that the new concept of maximal dissipative solutions may also inspire stable numerical schemes for the Navier--Stokes or Euler equations. Especially since the idea of maximal dissipation provides  a descent selection criterion for the approximation of turbulent flows. 


The proposed solution concept is very general and may be applied to various kinds of problems, we want to introduce  the concept here for a general system, but the main idea is to apply it to  the Navier--Stokes and Euler equations. 
But it can be applied in the sense of Definition~\ref{def:maxdiss} (below) to other systems featuring the relative energy inequality like systems in complex fluids like nematic liquid crystals~\cite{sabine}, models in phase transition~\cite{lrs} or~\cite{nsch}, or more generally GENERIC systems~\cite{generic}. 

Plan of the paper: 
First, we introduce the concept of  maximal dissipative solutions for a general class of dissipative isothermal systems (see~\eqref{eqgen} below) and collect some preliminary material. 
Supposing the existence of dissipative solutions, we prove the well-posedness of maximal dissipative solutions under some general assumptions. 
In Section~\ref{sec:nav}, we show that weak solutions to the Navier--Stokes equations are indeed dissipative solutions and apply the general result to infer well-posedness of maximal dissipative solutions to the Navier--Stokes equations. 
In Section~\ref{sec:eul}, we apply the general result of Theorem~\ref{thm:gen} to the Euler equations to infer well-posedness and compare this concept to measure-valued solutions. 

\section{Dissipative solutions\label{sec:gen}}
This section is devoted to a general approach to dissipative solutions. We consider a system, which can be seen as an isothermal GENERIC system~\cite{generic} and demonstrate the general scheme of dissipative solutions. 
\subsection{Relative energy inequality for general isothermal GENERIC systems}
We suppose the following  
\begin{itemize}
\item[\textbf{(A1)}]
We consider an energy functional $\mathcal{E}: \mathbb{V} \ra \R_+$ defined on a Banach space, which is assumed to be convex, coercive, and twice Gateaux differentiable with
\begin{align*}
D \mathcal{E} : \mathbb{V} \ra \mathbb{V}^* \quad \text{and} \quad D^2 \mathcal{E}: \mathbb{V } \ra  L( \mathbb{V}, \mathbb{V}^*)\,,
\end{align*}
where $L$ denotes the space of linear operators.

We suppose that $\f K : \mathcal{D}(\f K ) \subset \mathbb V^* \ra \mathbb{V}^{**}$ is a monotone, radial continuous, coercive operator, $ \f L : \mathcal{D}(\f L ) \subset \mathbb V \ra L( \mathbb V^*, \mathbb{V}^{**})$  denotes a skew-symmetric operator with $\f L(\f u) ^* = - \f L(\f u)$ such that $ \langle D\mathcal{E}(\f u) , \f L(\f u) D\mathcal{E}(\f u )\rangle = 0 $ for all $\f u \in \mathbb V$, and $\f f\in \mathbb{V}^{**} $ denotes a right-hand side. 
Note that $\f K$ and $\f L$ are possibly only densely defined, where $ \mathcal{D}$ denotes the domain.

\end{itemize}
 
We consider the following class of evolutionary problems
\begin{align}
\t \f u + \f K (D \mathcal{E} (\f u) ) = \f L(\f u) D\mathcal{E}(\f u) + \f f \,,\label{eqgen}
\end{align}

The considered system fulfills an energy dissipation mechanism, \textit{i.e.}, formally testing~\eqref{eqgen} by $D \mathcal{E}(\f u)$ provides
\begin{align}
\mathcal{E}(\f u (t) ) + \int_0^t \langle\f K (D \mathcal{E} (\f u) ) , D \mathcal{E}(\f u) \rangle   \de s \leq \mathcal{E}(\f u_0)+ \int_0^t \langle \f f , D \mathcal{E}(\f u) \rangle \de s  \quad \text{for a.e. }t\in (0,T)\,.\label{genen}
\end{align}
%
%
%
%
%
To abbreviate, we define the set of sufficiently regular function $\mathbb Y $ to be given by $\mathbb Y:= \mathcal{D}(\f K )  \cap \mathcal{D}(\f L ) $.


The formulation of  dissipative solutions follows a simple idea. Instead of formulating the equation~\eqref{eqgen} in a generalized way, we want to take the variation of the energy dissipation principle with respect to arbitrary functions, which do not have to be solutions of~\eqref{eqgen}. 
The energy-dissipation principle may be seen as the more important structure of the system. This means, that the equation is only formally derived from physical principles like the energy dissipation mechanism by following different approaches and often assuming a certain regularity on the hypothetical solution. It has been frequently observed (see~\cite{buckmaster}), that when this regularity is not present, the equation may describe something unphysical. 
Therefore sticking to the energy dissipation principle seems to be reasonable. 
We define the relative energy $\R : \mathbb{V}\times \mathbb{V} \ra \R_+$ and the relative dissipation $\mathcal{W}: \mathbb{V}\times \mathbb{V}\ra \R_+$ as variation of the energy and dissipation, respectively,  to be given by 
\begin{align*}
\mathcal{R}(\f u|\tu) ={}& \mathcal{E}(\f u) - \mathcal{E}(\tu) - \left \langle D  \mathcal{E}(\tu), \f u - \tu\right \rangle \quad \intertext{and} \quad \mathcal{W}(\f u|\tu) ={}& \frac{1}{2}\left \langle \f K ( D \mathcal{E}(\f u) ) - \f K ( D \mathcal{E}(\tu)), D \mathcal{E}(\f u)  - D \mathcal{E}(\tu) \right \rangle \,,
\end{align*}
respectively.
Since $\mathcal{E}$ is convex, $\mathcal R$ is nonnegative for all $\f u$, $\tu \in \mathbb{V}$ and since $\mathcal{E}$ is weakly lower semi-continuous and Gateaux differentiable on $\mathbb V$, the same holds for the mapping $\f u \mapsto \mathcal{R}(\f u | \tu)$ for all $\tu\in \mathbb{V}$. 
Similar, the monotony of the operator $ \f K$ guarantees that $\mathcal{W}$ is nonnegative for all $\f u$, $\tu \in \mathbb{V}$. 
\begin{remark}
In case that the dissipation operator $\f K$ is a potential operator and stems form a convex, lower semi-continuous potential $\Psi$, we may use standard convex analysis (see~\cite{ekeland}) to rewrite $\mathcal{W}$. 
Indeed, for a potential operator $ \f K$, we observe that 
\begin{align*}
\langle \f K( \f w ) , \f w) = \Psi(\f w) + \Psi^*  ( \f K(\f w)) \,,
\end{align*}
where $^*$ denotes the convex conjugate. 
Then $\mathcal{W}$ could be interpreted as the first Taylor approximation of 
$\Psi$ and $\Psi^*$,
\begin{align*}
\mathcal{W}( \f u | \tu) ={}& \Psi( D\mathcal{E}(\f u)) + \Psi^*(\f K(D \mathcal{E}(\f u))) -\left (\Psi( D\mathcal{E}(\tu)) + \Psi^*(\f K(D \mathcal{E}(\tu))) \right ) \\&- \left \langle \f K( D\mathcal{E}(\tu)), D\mathcal{E}(\f u) - D \mathcal{E}(\tu) \right \rangle - \left \langle D\mathcal{E}(\tu) , \f K( D\mathcal{E}(\f u) ) - \f K(D \mathcal{E}(\tu)) \right \rangle 
\\
={}& \Psi( D\mathcal{E}(\f u)) -\Psi( D\mathcal{E}(\tu))  - \left \langle D \Psi ( D\mathcal{E}(\tu)), D\mathcal{E}(\f u) - D \mathcal{E}(\tu) \right \rangle \\&
 + \Psi^*(\f K(D \mathcal{E}(\f u))) - \Psi^*(\f K(D \mathcal{E}(\tu)))  - \left \langle D\Psi^* ( \f K(D \mathcal{E}(\tu)))  , \f K( D\mathcal{E}(\f u) ) - \f K(D \mathcal{E}(\tu)) \right \rangle 
\,.
\end{align*}
Both lines are positive, as long as $\Psi$ and $\Psi^*$ are convex. The Gateaux derivative could easily be replaced by a subdifferential. Note that $ D\Psi^* ( \f K(\f w )) = \f w $, if $ \f K = D \Psi $. 

\end{remark}

We are now in the position to derive the relative energy inequality, at least formally.
Assume that $\f u $ is a sufficiently regular solution to~\eqref{eqgen} and $\tu$ be a general sufficiently smooth function.
We define the solution operator $\mathcal{A}: \mathbb{Y} \ra \mathbb V^{**}$ to be 
\begin{align*}
 \mathcal{A}(\tu ) = \t \tu + \f K( D \mathcal{E}(\tu)) - \f L( \tu) D \mathcal{E}(\tu) - \f f \,.
\end{align*}
Then it holds
\begin{align*}
\mathcal{R}( \f u | \tu) \Big |_0^t \leq{}& - \int_0^t \langle D\mathcal{E}(\f u ) , \f K ( D\mathcal{E}(\f u )) \rangle - \langle D\mathcal{E}(\tu ) , \f K ( D\mathcal{E}(\tu )) \rangle + \langle \mathcal{A}(\tu), D\mathcal{E}(\tu)\rangle \de s 
\\
&+\int_0^t \langle \f f , D\mathcal{E}(\f u) - D \mathcal{E}(\tu) \rangle \de s - \int_0^t \langle D^2\mathcal{E}(\tu) \t \tu , \f u - \tu \rangle + \langle D\mathcal{E}(\tu), \t \f u-\t \tu\rangle \de s 
\\
={}& 
- \int_0^t \langle D\mathcal{E}(\f u ) , \f K ( D\mathcal{E}(\f u )) \rangle - \langle D\mathcal{E}(\tu ) , \f K ( D\mathcal{E}(\tu )) \rangle+ \langle \mathcal{A}(\tu), D\mathcal{E}(\tu)\rangle \de s 
\\
&+\int_0^t \langle \f f , D\mathcal{E}(\f u) - D \mathcal{E}(\tu) \rangle \de s - \int_0^t \langle D\mathcal{E}(\tu), \t \f u-\t \tu\rangle+\langle \t \tu ,  D\mathcal{E}(\f u) - D \mathcal{E}(\tu) \rangle \de s 
\\
& + \int_0^t \langle \t \tu ,  D\mathcal{E}(\f u) - D \mathcal{E}(\tu) - D^2\mathcal{E}(\tu) (\f u - \tu )\rangle \de s 
\\={}& 
- \int_0^t \langle D\mathcal{E}(\f u ) , \f K ( D\mathcal{E}(\f u )) \rangle - \langle D\mathcal{E}(\tu ) , \f K ( D\mathcal{E}(\tu )) \rangle \de s
\\&  + \int_0^t  \langle D\mathcal{E}(\tu), \f K( D \mathcal{E}(\f u) ) - \f K( D \mathcal{E}(\tu)) -  \left ( \f L(\f u)  D\mathcal{E}(\f u)- \f L(\tu)  D\mathcal{E}(\tu)\right  )\rangle \de s 
\\ & + \int_0^t \langle \f K (D\mathcal{E}(\tu) ) - \f L(\tu)D\mathcal{E}(\tu)  ,D\mathcal{E}(\f u) - D \mathcal{E}(\tu)\rangle -\langle \mathcal{A}(\tu)  ,D\mathcal{E}(\f u) - D \mathcal{E}(\tu)  \rangle \de s 
\\&+ \int_0^t \langle \t \tu ,  D\mathcal{E}(\f u) - D \mathcal{E}(\tu) - D^2\mathcal{E}(\tu) (\f u - \tu )\rangle \de s 
\\={}& 
- \int_0^t \langle D\mathcal{E}(\f u ) - D \mathcal{E}(\tu)  , \f K ( D\mathcal{E}(\f u )) -   \f K ( D\mathcal{E}(\tu )) \rangle \de s 
\\&  - \int_0^t   \langle D\mathcal{E}(\tu),   \left ( \f L(\f u) - L(\tu)\right ) \left (   D\mathcal{E}(\f u)- \  D\mathcal{E}(\tu)\right  )\rangle+\langle \mathcal{A}(\tu)  , D\mathcal{E}(\f u)- \  D\mathcal{E}(\tu)\rangle  \de s\\&+ \int_0^t \langle \t \tu ,  D\mathcal{E}(\f u) - D \mathcal{E}(\tu) - D^2\mathcal{E}(\tu) (\f u - \tu )\rangle \de s 
 \,.
\end{align*}
The inequality in the previous calculation is due to the energy inequality~\eqref{genen}. Additionally, we added and subtracted the solution operator $\mathcal{A}(\tu) $ tested with $D\mathcal{E}(\tu)$ in the first step. The second step is just a reformulation, the third uses the fact that $\f u $ solves~\eqref{eqgen} and adding and subtracting the solution operator $\mathcal{A}(\tu) $ accordingly. The last step is again a rearrangement using the fact that $\f L$ is skew-symmetric. 

This is a typical way to calculate the relative energy, in order to estimate the right-hand side, we have to assume that 
\begin{itemize}
\item[\textbf{(A2)}]
There exists a form $\mathcal{K}: \mathbb{Y}\ra \R_+ $ such that the estimate 
\begin{multline}
  |\langle  \t \tu   ,D\mathcal{E}(\f u) - D\mathcal{E}(\tu)-  D^2\mathcal{E}(\tu) (\f u - \tu )\rangle |
 +|    \langle D\mathcal{E}(\tu),   \left ( \f L(\f u) - L(\tu)\right ) \left (   D\mathcal{E}(\f u)- \  D\mathcal{E}(\tu)\right  )\rangle | \\
\leq \mathcal{K}(\tu) \mathcal{R}(\f u | \tu) +\mathcal{W}(\f u | \tu)  
\label{condition}
\end{multline}
holds.
\end{itemize}
 We find
\begin{multline*}
\mathcal{R}(\f u( t) | \tu(t) ) + \int_0^t  \mathcal{W}(\f u | \tu )   + \left \langle \mathcal{A}(\tu) , D\mathcal{E}(\f u)- \  D\mathcal{E}(\tu) \right \rangle   \de s \\
\leq \mathcal{R}(\f u_0| \tu (0) ) + \int_0^t \mathcal{K}(\tu) \mathcal{R}(\f u | \tu) \de s \,.
\end{multline*}
The relative energy inequality is resulting from applying Gronwall's inequality
\begin{align*}
\mathcal{R}(\f u( t) | \tu(t) ) + \int_0^t \left ( \mathcal{W}(\f u|  \tu )   + \left \langle \mathcal{A}(\tu , D\mathcal{E}(\f u)- \  D\mathcal{E}(\tu) \right \rangle \right )  e^{\int_s^t\mathcal{K}(\tu) \de \tau } \de s \leq \mathcal{R}(\f u_0| \tu(0) ) e^{\int_0^t \mathcal{K}(\tu)\de s } \,.
\end{align*}

\begin{remark}
The presented calculations are only formal and should only demonstrate a general strategy how to derive a relative energy inequality for the considered class of equations. The considered assumptions could be generalized in several fashions. In a sense, the presented approach represents a generalization of the gradient flow approach to dissipative Hamiltonian systems and some of the generalizations in the gradient flow setting (see for instance~\cite{mielke}) could also be applied here. 

There is a lot of freedom, when formulating a relative energy for instance in the choice of $\mathcal{K}$ or $\mathcal{W}$. 
Concerning the choice of $\mathcal{K}$, this regularity criterion usually has to be sufficient to provide uniqueness of weak solutions. For Navier--Stokes, we define it according to Serrin's uniqueness criterion. But it may also be chosen differently, such that the emerging maximal dissipative solution differs and depends on the choice of $\mathcal{K}$. 
Also the choice of $\mathcal{W}$ has some freedom of choice. There it is desirable that $\mathcal{W}$ is weakly-lower semi-continuous and even convex. 

The condition~\eqref{condition} gives some condition on the continuity of $\f L$ with respect to the energy and dissipation. And some condition on the regularity of the energy. There are different formulations possible, depending on the considered case.

\end{remark}
A recurrent assumption on the energy $\mathcal{E}$ is given by
\begin{itemize}
\item[\textbf{(A3)}] There exists a space $\mathbb{Z} \supset \mathbb V$ such that for some constant $C>0$ it holds
\begin{align}
\| D \mathcal{E}(\f u) \|_{\mathbb{Z}} \leq C ( \mathcal{E}(\f u ) +1)  \qquad \text{for all }\f u \in \mathbb V\label{enest}\,.
\end{align}
\end{itemize}
Under Condition \textbf{(A1)}, for every right hand side  $\f f
\in L^1(0,T; \mathbb{Z}^*)$ one may deduce \textit{a priori} estimates from~\eqref{genen} with the inequality $ |\langle \f f , D \mathcal{E}(\f u) \rangle | \leq \| \f f \|_{\mathbb Z^*} \| D \mathcal{E}(\f u) \| _{\mathbb Z} $, inequality~\eqref{enest}, and Gronwall's Lemma. For $\mathcal{E} $ locally sufficiently regular and fulfilling Assumption \textbf{(A1)}, the inequality 
\begin{align*}
\| D\mathcal{E}(\f u) - D\mathcal{E}(\tu)-  D^2\mathcal{E}(\tu) (\f u - \tu )\|_{\mathbb Z } \leq c \mathcal{ R} (\f u | \tu) \,
\end{align*}
holds for every $ \f u \in \mathbb V $ and $\tu \in \mathbb Y$ (see~\cite[Sec.~4.3]{lrs}).
Note that this inequality is trivially fulfilled, if $\mathcal{E}$ is quadratic, since the left-hand side vanishes in this case. 


\subsection{Definitions and general result}
We may define the dissipative and maximal dissipative solution concept and prove a general well-posedness result for maximal dissipative solutions assuming that at least one dissipative solution exists.  
\begin{definition}[dissipative solution]\label{def:diss} 
A function $\f u$ is called a dissipative solution, if $ \f u \in \mathbb{X}$ and the relative energy inequality
\begin{multline}
\mathcal{R}(\f u( t) | \tu(t) ) + \int_0^t \left ( \mathcal{W}(\f u , \tu )   + \left \langle \mathcal{A}(\tu (t)) , D\mathcal{E}(\f u)- \  D\mathcal{E}(\tu) \right \rangle \right )  e^{\int_s^t\mathcal{K}(\tu) \de \tau } \de s \\
\leq \mathcal{R}(\f u_0|  \tu(0) ) e^{\int_0^t \mathcal{K}(\tu)\de s } \label{relen}
\end{multline}
holds for a.e.~$t\in (0,T)$ and for all $ \tu \in \C^1([0,T];\mathbb{Y})$. 
\end{definition} 
\begin{remark}[Regular dissipative solutions]\label{rem:disreg}
Dissipative solutions fulfill the so-called weak-strong uni\-que\-ness property. They coincide with a strong solution emanating from the same initial datum, as long as the latter exists. Indeed, let $\bar{\f u}$ be a strong solution. Then it can be inserted as a test function in~\eqref{relen}. Since $\mathcal{A}( \bar{\f u}) = 0$ and $\mathcal{R}(\f u_0 | \bar{\f u}(0)) =0$. the inequality~\eqref{relen} can only be satisfied for $\tu = \bar{\f u}$, if $\f u =\bar{\f u }$. 

Additionally, 
it holds that if there exists a regular dissipative solution, then this solution is a regular weak solution, \textit{i.e.}, a strong solution. 
Indeed, assume that the dissipative solution $\bar{\f u}$ is regular, \textit{i.e.}, $\bar{\f u} \in \mathcal{Y}$, then also $ \tilde{\f u} = \bar{\f u} + \alpha \f r \in \mathcal{Y}$ for $\f r \in \C^\infty_c(\Omega \times [0,T] ; \R^d)$ with $\alpha > 0 $ sufficiently small. 
Inserting $ \tilde{\f u} = \bar{\f u} + \alpha \f r $  into~\eqref{relen} for the dissipative solution $\f u = \bar{\f u}$ and dividing by $\alpha$, we end up with
\begin{align*}
o(\alpha) \leq \int_0^t    \left \langle \mathcal{A}(\bar{\f u})  , 
\f r 
\right \rangle e^{\int_s^t  \mathcal{K} (\bar{\f u}) \de \tau+ o(\alpha) } \de s  + o(\alpha) \,,
\end{align*}
where $o (\alpha ) \ra 0 $ for $\alpha \ra 0$,
since the only linear term in $\alpha$ occurs in the last term on the right-hand side of~\eqref{relen} and all other appearing terms are super-linear in $\alpha$. 
Passing to the limit $\alpha\ra 0$ implies that 
\begin{align*}
0 \leq  \int_0^t \left \langle \mathcal{A}(\bar{\f u}) , 
\f r\right \rangle \de s\,.
 \end{align*}
the above inequality is in fact an equality  (since $\f r$ was arbitrary) and hence, $\bar{\f u}$  fulfills a standard weak formulation.  
\end{remark}

\begin{definition}[maximal dissipative solution]\label{def:maxdiss}
A function $\f u$ is called a maximal dissipative solution, if $ \f u \in \mathbb{X}$   is the solution of the following optimization problem
\begin{align*}
\min_{\f u \in \mathbb{X}} \frac{1}{2} \int_0^T \mathcal{E}( \f u( t) )  \de t \quad \text{such that $\f u $ is a dissipative solution according to Definition~\ref{def:diss}. }
\end{align*}
\end{definition} 

In the following, we introduce certain assumptions under which we will prove a general theorem. These assumptions are not very general, we mainly have  the application to the Navier--Stokes and Euler equations in mind. Afterwards, we may comment on possible generalizations. 
\begin{itemize}
\item[\textbf{(A4)}] 
Let $\mathcal{E}: \mathbb V \ra \R$ be a quadratic energy such that $D\mathcal E \in L (\mathbb V , \mathbb V^* )$ is a linear operator.  
We assume that $\f K$ is of polynomial growth, \textit{i.e.}, there exists a Banach space $ \mathbb W \subset \mathbb V $ such that there exists $c_{\f K}$ and $C_{\f K}>0$ with  $$c_{\f K} \left (  \| \f w \| _{\mathbb{W}} ^p-1 \right ) \leq \langle \f K( \f w) , \f w\rangle \leq C_{\f K} \| \f w \|_{\mathbb W}^p+1 \,.$$ 
Additionally, the mapping $ \f u \mapsto \mathcal{W}(\f u | \tu)$ is convex and weakly-lower semi-continuous for every $\tu \in \mathbb Y$. 
\end{itemize}
We note that the mapping $ \f u \mapsto \mathcal{R}(\f u | \tu)$ is already known to be convex and weakly-lower semi-continuous for every $\tu \in \mathbb Y$ due to Assumption \textbf{(A1)}. The Assumption~\textbf{(A3)} is trivially fulfilled under the Assumption~\textbf{(A4)} for $\mathbb{Z}= \mathbb{V}^*$. 
From the energy estimate~\eqref{genen}, we observe that the natural state space is given by 
 $ \f u \in \mathbb X$, where  $ \f u \in \mathbb X$, if $  \f u \in L^\infty(0,T;\mathbb V)$ and  $D \mathcal{E}(\f u) \in  L^p(0,T;\mathbb W)$.  In the case of $\f K \equiv 0$, the natural state space is given by  $\mathbb{X}:= L^\infty(0,T;\mathbb V)$.
The associated space for the right-hand side is then given by  $ \f f \in L^{p'} (0,T;\mathbb{W}^*) \bigoplus L^1(0,T; \mathbb{Z}^*)$ for $p'=p/(p-1)$. 
For such a right-hand side, we may deduce \textit{a priori} estimates from the energy estimate~\eqref{genen}. Indeed, for $\f f \in L^{p'} (0,T;\mathbb{W}^*) \bigoplus L^1(0,T; \mathbb{Z}^*)$ there exists a 
$\f f_1 \in L^{p'} (0,T;\mathbb{W}^*) $ and a $\f f_2 \in L^1(0,T; \mathbb{Z}^*) $ such that $\f f = \f f _1 + \f f _2 $. 
This allows to estimate by Young's inequality and~\eqref{enest} for  $\mathbb{Z}= \mathbb{V}^*$.
\begin{align*}
\langle \f f , D \mathcal E (\f u ) \rangle = \langle \f f _1, D \mathcal E (\f u ) \rangle +\langle \f f _2, D \mathcal E (\f u ) \rangle \leq \frac{c_k}{2} \left \|  D \mathcal E (\f u ) \right \| _{\mathbb W} ^p + C \| \f f_1 \|_{\mathbb W^*} ^{p'} +C \| \f f_2 \|_{\mathbb Z^*}  \left ( 1 + \mathcal{E}( \f u ) \right ) \,,
\end{align*}
 which provides \textit{a priori} estimates in $\mathbb X$ when inserted into~\eqref{genen}.

The convergence $\f u^n  \stackrel{*}{\rightharpoonup} \f u $ in $\mathbb{X} $ means that there exists a $\f u \in \mathbb X$ such that
\begin{align*}
\f u^n \stackrel{*}{\rightharpoonup} \f u \quad \text{in } L^\infty(0,T;\mathbb V)  \qquad \text{and} \qquad D\mathcal{E}(\f u^n ) \rightharpoonup D\mathcal{E}(\f u) \quad \text{in } L^p(0,T;\mathbb W)\,.
\end{align*}
Due to the linearity of $D\mathcal E $, the first convergence implies the second one.

\begin{theorem}\label{thm:gen}
  Let $\mathcal{R}$, $\mathcal{W}$, $\mathcal{K}$, and $\mathcal{A}$ be given as above and let the assumption \textbf{(A1)}, \textbf{(A3)}, and \textbf{(A4)} be fulfilled.
Then the set of dissipative solutions   is closed and convex. In case that there exists a dissipative solution according to~Definition~\ref{def:diss} for any initial datum $\f u _0 \in \mathbb V$ and right-hand side $\f f \in  L^{p'} (0,T;\mathbb{W}^*) \bigoplus L^1(0,T; \mathbb{V}^*)$, then there exists a unique maximal dissipative solution $\f u \in \mathbb{X}$
  in the sense of Definition~\ref{def:maxdiss} and the maximal dissipative solution depends continuously on the initial datum and the right-hand side in the sense
$( \f u ^n_0 , \f f^n ) \ra ( \f u _0 , \f  f ) $ in $ \mathbb{V} \times ( L^{p'} (0,T;\mathbb{W}^*) \bigoplus L^1(0,T; \mathbb{Z}^*)) $, then  to every $n\in \N$ there exists a maximal dissipative solution $\f u^n \in \mathbb{X}$ to given values $(\f u_0^n , \f f ^n)$ and it holds 
$ \f u^n \stackrel{*}{\rightharpoonup} \f u $ in $\mathbb{X}$.
 \end{theorem}
 
 \begin{remark}[Generalizations of assumptions]
 The assumption \textbf{(A4)} is very much tailored to the needs of the Navier--Stokes and Euler equations. Especially the energy is restricted to the quadratic case. 
This can be further generalized. If $D\mathcal E$ is not linear, several adaptations are possible. In case that $D\mathcal E $ is still invertible, a good state space is rather the one of $D\mathcal{E}$, \textit{i.e.}, $\mathbb W \subset \mathbb V^*$. Then, convexity and weakly lower semi-continuity have to be  assumed for the mappings $ \f w \mapsto \mathcal R( (D\mathcal E)^{-1}(\f w) , \tu ) $ and $   \f w \mapsto \mathcal W( (D\mathcal E)^{-1}(\f w) , \tu ) $ in order to deduce that the set of dissipative solutions is closed and convex. 
 It would also be possible to prove a general existence result via some Galerkin approximation in the natural state space, but we refrained from executing it here. Often existence of dissipative solutions is already known and a descent discretization depends heavily on the specific features of a considered system.

 The estimate in Assumption \textbf{(A2)} only holds for rather general energies fulfilling also estimate~\eqref{enest}. It is also possible to extend this approach to more irregular energies that are only proper and convex and do not fulfill~\eqref{enest}. We are going to address this issue in a subsequent publication. 
 \end{remark}

 Before we prove the above theorem, we provide a preliminary lemma. 
\begin{lemma}\label{lem:invar}
Let $f\in L^1(0,T)$ and $g\in L^\infty(0,T)$ with $g\geq 0$ a.e.~in $(0,T)$.
Then the two inequalities 
\begin{align*}
-\int_0^T \phi'(t) g(t) \de t - \phi(0) g(0) + \int_0^T \phi(t) f(t) \de t \leq 0 
\end{align*}
for all $ \phi\in\C^\infty_c([0,T)) ${ with } $\phi \geq 0 $, and $ \phi'\leq 0$ on $[0,T]$
and 
\begin{align}
g(t) -g(0) + \int_0^t f(s) \de s \leq 0 \quad \text{for a.e.~}t\in(0,T)\,\label{ineq2}
\end{align}
are equivalent.
\end{lemma}
\begin{proof}
The proof of the first implication is a simple adaptation of the classical variational lemma (compare to~\cite{weakstrong}).
It can be seen, by choosing a sequence $\{ \phi_\varepsilon\}$ as a suitable (monotone decreasing) approximation  of the indicator function on $[0,t]$, \textit{i.e.}, $\chi_{[0,t]}$, with the condition $\phi_\varepsilon(0)=1$. 

The reverse implication can be seen, by testing~\eqref{ineq2} by $- \phi'$ and integrate-by-parts in the last two terms. 

\end{proof}

\begin{remark}[Reformulation of the optimization problem]
The minimization problem of Definition~\eqref{def:maxdiss} can be seen as a optimization problem with infinitely many inequality constraints. Therefore, it can be rewritten (using Lemma~\ref{lem:invar}) in the following way
\begin{align*}
 \f u =\argmin_{\f u\in \mathbb {X}} \left ( \int_0^T \mathcal{E}(\f u (t) ) \de t +  \mathcal{F}(\f u )\right )  \,,
 \end{align*}
 where
\begin{multline*}
\mathcal{F}(\f v) = {}
%
 \sup _{\phi\in \tilde{\C} (0,T)}\sup _{\tu\in \C^1([0,T];\mathbb{Y})}\Big( -\int_0^T \phi'(t) \mathcal{R}(\f u(t)|\tu(t)) e^{-\int_0^t\mathcal{K}(\tu(s))\de s }  \de t - \phi(0) \mathcal{R}(\f u _0,\tu (0))  \\{}+ \int_0^T \phi(t) \left ( \mathcal{W}(\f u(t) | \tu (t)) + \left \langle \mathcal{A}(\tu (t)) ,D \mathcal E(\f u (t))-D \mathcal E(\tu(t))  \right \rangle \right )e^{-\int_0^t\mathcal{K}(\tu (s))\de s }    \de t\Big)\,,
\end{multline*}
where $\phi \in \tilde{\C} ([0,T])$ as long as  $ \phi\in\C^1([0,T]) ${ with } $\phi \geq 0 $, and $ \phi'\leq 0$ on $[0,T]$ as well as $\phi(0)=1$ and $\phi(T)=0$. 

Therefore, inequality~\eqref{relen} is multiplied by $e^{-\int_0^t\mathcal{K}(\tu) \de \tau}$, the supremum is taken over the side conditions, \textit{i.e.}, all $\tu \in \mathbb {Y}$. Finally, the supremum is taken over the Lagrange multiplier $\phi$. 
\end{remark}

\begin{proof}[Proof of Theorem~\ref{thm:gen}]
The proof is divided in several steps: \textit{Step 1, Solution set is closed:}
Let $\{ \f u^n  \} $ be sequence of dissipative solutions according to Definition~\ref{def:diss}. Since the relative energy inequality~\eqref{relen} also holds for $\tu = 0$, we may infer the energy inequality~\eqref{eninv}.   This leads to \textit{a priori} estimates and let us deduce the standard weak convergence
\begin{align}
\f u^ n \stackrel{*}{\rightharpoonup} \f u \,,\quad \text{in } \mathbb{X} \,.\label{convergence}
\end{align}

Considering the 
relative energy inequality~\eqref{relen} for $\f u^n$ and  fixed $\tu$, we 
multiply it by $e^{-\int_0^t\mathcal{K}(\tu) \de \tau}$ and use Lemma~\ref{lem:invar} to infer
\begin{align}
\begin{split}
-&\int_0^T \phi'(t) \mathcal{R}(\f u^n(t)|\tu (t)) e^{-\int_0^t\mathcal{K}(\tu (s))\de s }  \de t  -\phi(0) \mathcal{R}(\f u^n _0,\tu (0))\\&{}+ \int_0^T \phi(t) \Big ( \mathcal{W}(\f u^n(t) | \tu (t)) + \left \langle \mathcal{A}(\tu (t)) ,D\mathcal{E}(\f u^n(t))-D \mathcal{E}(\tu(t))  \right \rangle \Big )e^{-\int_0^t\mathcal{K}(\tu (s))\de s }    \de t\leq  0 \,. \end{split}
\label{relenweak}
\end{align}
for all  $\phi \in \tilde{\C} ([0,T])$. 

In this formulation, we may pass to the limit since 
 $\mathcal{R}(\cdot | \tu )$ and $\mathcal{W}(\cdot | \tu )$ are weakly lower semi-con\-tin\-uous with respect to the convergence~\eqref{convergence} and $D\mathcal{E}(\f u^n) $ occurs linear multiplied with the solution operator $\mathcal{A}(\tu )$ such that weak convergence is sufficient to pass to the limit in this term. The inequality~\eqref{relenweak} also holds in the limit for every $\vv\in \C^1([0,T]; \mathbb Y) $ such that we deduce from Lemma~\ref{lem:invar} that~\eqref{relen} also holds for the limit~$\f u$ and hence, $\f u$ fulfills the Definition~\ref{def:diss}.

\textit{Step 2: Convexity of the solution set:}
The convexity of the solution set  follows again from the fact that the relative energy inequality is convex in $\f v$.  
Indeed, let $\f v_1$ and $\f v_2$ be two solutions in the sense of Definition~\ref{def:diss}. 
Since $\mathcal{R}(\cdot | \vv )$ and $\mathcal{W}(\cdot | \vv)$ are convex for fixed $\vv\in \mathbb Y $, we may deduce
\begin{multline*}
\mathcal{R}(\lambda\f v_1 + (1-\lambda)\f v_2 | \vv ) \\+ \int_0^t \left ( \mathcal{W}(\lambda\f v_1 + (1-\lambda)\f v_2 , \vv )   + \left \langle \mathcal{A}(\vv ) ,D \mathcal{E}(\lambda\f v_1 + (1-\lambda)\f v_2)-D \mathcal{E}(\vv) \right \rangle \right )  e^{\int_s^t\mathcal{K}(\vv) \de \tau } \de s\\ 
\leq 
\lambda \left ( \mathcal{R}(\f v_1 | \vv ) + \int_0^t \left ( \mathcal{W}(\f v_1 , \vv )   + \left \langle \mathcal{A}(\vv ) ,D \mathcal{E}(\f v_1 )-D \mathcal{E}(\vv)  \right \rangle \right )  e^{\int_s^t\mathcal{K}(\vv) \de \tau } \de s \right )\\ 
+ (1-\lambda) \left ( \mathcal{R}(\f v_2 | \vv ) + \int_0^t \left ( \mathcal{W}(\f v_2 , \vv )   + \left \langle \mathcal{A}(\vv ) ,D \mathcal{E}(\f v_2)-D \mathcal{E}(\vv) \right \rangle \right )  e^{\int_s^t\mathcal{K}(\vv) \de \tau } \de s\right ) \\
 \leq \left ( \lambda + (1-\lambda ) \right  ) \mathcal{R}(\f v_0| \vv(0) ) e^{\int_0^t \mathcal{K}(\vv)\de s }
\end{multline*}
for all $\lambda\in [0,1]$, since $\f v_1 $ and $\f v_2$ are assumed to be dissipative solutions and $D\mathcal{E}$ is a linear map.   
This implies that $\lambda\f v_1 + (1-\lambda)\f v_2$ is also a  dissipative solution. 

 \textit{Step 3, Well posedness:} 
 First, we have to check whether the solution concept is well-posed. To this end, we observe that the solution set of dissipative solutions is closed and convex according to the previous two steps. 
 Since the functional $\f u \mapsto  \frac{1}{2} \int_0^T \| \f u \|_{L^2(\Omega)}^2 \de t $ is a strictly convex lower semi-continuous functional defined on this closed and convex solution set and thus, has a unique minimizer~\cite[Prop.~1.3.1]{RoubicekMeasure}
 
 \textit{Step 4, Continuous dependence: }
 Then, we consider the perturbed problem, where $\f u_0$ and $ \f f$ are replaced by $\f u_0^n $ and $\f f ^n$, respectively. 
 We assume that the initial values and right-hand sides converge $( \f u ^n_0 , \f f^n ) \ra ( \f u_0 , \f  f ) $ in $ \mathbb{V} \times   L^{p'} (0,T;\mathbb{W}^*) \bigoplus L^1(0,T; \mathbb{Z}^*)$.
First, we observe that to every pair $  ( \f u ^n_0 , \f f^n ) \in \mathbb{V} \times  L^{p'} (0,T;\mathbb{W}^*) \bigoplus L^1(0,T; \mathbb{Z}^*) $ there exists a unique maximal dissipative solution $\f u^n$. 
 Then, we may prove the continuity of the relative energy inequality. 
  Indeed, considering the difference of the relative energy inequality~\eqref{relen} for fixed $\f u$ and  $\tu $  multiplied by $e^{-\int_0^t\mathcal{K}(\tu )\de s }$ for two different pairs  $( \f u ^n_0 , \f f^n )$ and $( \f u _0 , \f  f )$, we find
  \begin{multline}
    \Big | \mathcal{R}( \f u | \tu )e^{-\int_0^t \mathcal{K}(\tu) \de s } + \int_0^t \left ( \mathcal{W}( \f u | \tu ) + \left \langle \mathcal{A}_{\f f}(\tu) , D \mathcal{E}(\f u) -D \mathcal{E}( \tu) \right \rangle \right )e^{-\int_0^s \mathcal{K}(\tu )\de \tau } \de s  \\
   -\left (  \mathcal{R}( \f u | \tu ) e^{-\int_0^t \mathcal{K}(\tu) \de s } + \int_0^t \left ( \mathcal{W}( \f u | \tu ) + \left \langle \mathcal{A}_{\f f^n}(\tu) , D \mathcal{E}(\f u) - D \mathcal{E}(\tu) \right \rangle \right )e^{-\int_0^s \mathcal{K}(\tu )\de \tau } \de s\right ) \Big | \\
+|\mathcal{R}( \f u_0 | \tu (0)     - \mathcal{R}( \f u^n_0 | \tu(0)| 
   \\
   \leq \int_0^t \left | \left \langle \f f(s) - \f f^n(s) , D \mathcal{E}(\f u(s)) - D \mathcal{E}(\tu(s)) \right \rangle \right | e^{-\int_0^s \mathcal{K}(\tu)\de \tau } \de s+ \mathcal{R}( \f u_0| \f u_0^n ) \,.
     \label{inconrel}
 \end{multline}
 We observe that $ \f u$ is bounded in $ \mathbb{X}$ due to the energy estimates and $e^{-\int_0^s \mathcal{K}(\tu )\de \tau } <1$. Additionally, $\tu (s) e^{-\int_0^s \mathcal{K}(\tu)\de \tau } $ is  bounded in $\mathbb{X} $ for $\tu \in \C^1([0,T]; \mathbb{Y})$. Thus, the strong convergences of the initial values and the right-hand sides $( \f u ^n_0 , \f f^n ) \ra ( \f u _0 , \f  f ) $ in $ \mathbb{V} \times   L^{p'} (0,T;\mathbb{W}^*) \bigoplus L^1(0,T; \mathbb{Z}^*)  $ allow to pass to the limit on the right-hand side of~\eqref{inconrel}. 
 Note that this convergence is independent of $\tu$ and hence uniform in $\tu$, it also holds, taking the supremum over $\tu\in \C^1([0,T]; \mathbb{Y}) $ with bounded energy and dissipation. 
Since the side-condition converge and the minimizing functional remains the same, the unique maximal dissipative solutions  $\f u^n$ converges to $\f u$ compare to~\cite[Prop.~1.3.2]{RoubicekMeasure}. Since all terms are bounded in $\mathbb{X}$, the convergence is also weak in $\mathbb{X}$.

\end{proof}

Throughout this paper, let $\Omega \subset \R^d$ be a Lipschitz domain with $d \geq  2$.
The space of smooth solenoi\-dal functions with compact support is denoted by $\mathcal{C}_{c,\sigma}^\infty(\Omega;\R^d)$. By $\f L^p_{\sigma}( \Omega) $, $\V(\Omega)$,  and $ \f W^{1,p}_{0,\sigma}( \Omega)$, we denote the closure of $\mathcal{C}_{c,\sigma}^\infty(\Omega;\R^d)$ with respect to the norm of $\f L^p(\Omega) $, $ \f H^1( \Omega) $, and $ \f W^{1,p}(\Omega)$ respectively.
Note that $\f L^2_{\sigma}(\Omega) $ can be characterized by $\f L^2_{\sigma}(\Omega) = \{ \f v \in L^2(\Omega)| \di \f v =0 \text{ in }\Omega\, , \f n \cdot \f v = 0 \text{ on } \partial \Omega \} $, where the first condition has to be understood in the distributional sense and the second condition in the sense of the trace in $H^{-1/2}(\partial \Omega )$. 
The dual space of a Banach space $V$ is always denoted by $ V^*$ and equipped with the standard norm; the duality pairing is denoted by $\langle\cdot, \cdot \rangle$ and the $L^2$-inner product by $(\cdot , \cdot )$. We use the standard notation $( \f H^1_0(\Omega))^*=\f H^{-1}(\Omega)$.
By $  \mathbb{M}^{d\times d}$ we denote $d$-dimensional quadratic matrices, by $  \mathbb{M}^{d\times d}_+$ the positive definite subset, and by $  \mathbb{M}^{d\times d}_{\sym,+}$ the symmetric positive definite matrices. The Radon measures taking values in a set $A \subset \R^n$ are denoted by $\mathcal{M}(\ov\Omega ; A ) $, which may be interpreted as the dual space of the continuous functions, \textit{i.e.,} $\mathcal{M}(\ov\Omega; A ) =(\C(\ov \Omega; A ) )^*$.


\section{Navier--Stokes equations\label{sec:nav}}
In this section, we apply the general result to the Navier--Stokes equations. 
\subsection{Weak solutions and dissipative solutions} 

First we recall the Navier--Stokes equations for the sake of completeness. 
\begin{align}
\begin{split}
\t \f v + ( \f v \cdot \nabla ) \f v - \nu \Delta \f v + \nabla p = \f f, \qquad \di \f v ={}& 0 \qquad \text{in }\Omega \times (0,T)\,,\\
\f v (0) ={}& \f v_0 \qquad \text{in } \Omega \,,\\
\f v = {}& 0 \qquad \text{on }\partial \Omega\times (0,T) \,.
\end{split}\label{nav}
\end{align}

The underlying spaces in the Navier--Stokes case are given by  $ \mathbb{X}= L^\infty(0,T; \Ha)\cap L^2(0,T;\V)$ and $\mathbb{Y}= C^1([0,T]; \f H^2 \cap \V )$. 
We define the relative energy $\mathcal{R}$ by
\begin{subequations}\label{def:nav}
\begin{align}
\mathcal{R}(\f v|\vv) ={}& \frac{1}{2}\| \f v -\vv \|_{L^2(\Omega)}^2 \,,
\intertext{the relative dissipation $\mathcal{W}$ by}
\mathcal{W}(\f v | \vv) ={}& \frac{\nu}{2} \| \nabla \f v-\nabla \vv \|_{L^2(\Omega)}^2 \,,
\intertext{the regularity measure $\mathcal{K}$ by}
\mathcal{K}(\vv)={}& \mathcal{K}_{s,r} (\vv)=c \| \vv\|_{L^{r}(\Omega)} ^{s} \quad \text{for } \frac{2}{s}+\frac{d}{r}\leq 1 \,,
\intertext{and the solution operator $\mathcal{A}$ by}
\langle \mathcal{A}(\vv) , \cdot \rangle ={}& \langle \t \vv + (\vv \cdot \nabla)\vv - \nu \Delta \vv  - \f f+ \nabla 
\tilde{p}, \cdot \rangle \,,\label{A}
\end{align}
\end{subequations}
which has to be understood in a weak sense, at least with respect to space. Note that the solution operator does not include boundary condition, since they are encoded in the underlying spaces. This may changes for different boundary conditions. 

First, we show that weak solutions to the Navier--Stokes equations are indeed dissipative solutions. 
The set of dissipative solutions can be seen as the convex envelope or super set of the set of weak solutions. 
The set of dissipative solutions is bigger, but it is easier to define descent selection criteria on this convex compact super set of weak solutions. 

 \begin{proposition}\label{prop:diss}
 Let $\f v \in \mathbb{X}$ be a weak solution to~\eqref{nav}. Then it is a dissipative solution. 
 \end{proposition}
 \begin{proof}
 Let $\f v $ be a weak solution to the Navier--Stokes equation with energy inequality. Then it fulfills the weak formulation
\begin{align}
-\int_0^ T \int_\Omega \f v \t \f \varphi  \de \f x \de t + \int_0^T \int_\Omega \left ( \nu \nabla \f v : \nabla \f \varphi + ( \f v \cdot \nabla ) \f v \cdot \f \varphi \right ) \de \f x \de t = \int_0^T \langle \f f , \f \varphi \rangle \de t + \int_\Omega \f v \cdot\f \varphi(0) \de \f x 
\label{weakv}
\end{align}
for $\f \varphi \in\C^1_c([0,T);\mathcal{C}_{c,\sigma}^\infty(\Omega;\R^3))$and the energy inequality
\begin{align}
\frac{1}{2}\| \f v(t) \| _{L^2(\Omega)}^2 +\nu \int_0^t \| \nabla \f v \|_{L^2(\Omega)}^2 \de s \leq \frac{1}{2}\| \f v_0 \|_{L^2(\Omega)}^2 + \int_0^t \langle \f f, \f v \rangle \de s \quad \text{for a.e.~}t\in(0,T) \,.\label{eninv}
\end{align}

For a test function $\vv\in  \C^1([0,T];\mathbb{Y})$, we find by testing the solution operator $\mathcal{A}(\vv)$ by $\phi \vv$ with $\phi \in \C^1_c([0,T))$ and standard calculations that
\begin{multline}
\int_0^T \phi \left \langle\mathcal{A}(\vv)\vv  \right\rangle   \de t =\\ -\int_0^T \phi '   \frac{1}{2} \| \vv (t) \|^2_{L^2(\Omega)} \de t   + \int_0^T \phi\left ( \nu  \| \nabla \f v \|^2_{L^2(\Omega)} - \langle \f f , \vv \rangle\right )  \de t - \phi(0) \frac{1}{2}\| \vv (0) \|_{L^2(\Omega)}^2  \label{eninvv}
\end{multline}
Testing again the solution operator $\mathcal{A}(\vv)$ by $\phi\f v$ and~\eqref{weakv} by $\phi \vv$ with $\phi \in \C^1_c([0,T))$, we find
\begin{multline}
- \int_0^T \phi' \int_\Omega \f v \cdot \vv \de \f x \de t + \int_0^T \phi\int_\Omega \left ( 2 \nu \nabla \f v : \nabla \vv + ( \f v \cdot \nabla ) \f v \cdot \vv + ( \vv \cdot \nabla) \vv \cdot \f v \right ) \de \f x \de t \\= \int_0^T \phi \left\langle  \mathcal{A}(\vv), \f v\right \rangle  \de t + \phi(0) \int_\Omega \f v_0 \cdot \vv(0) \de \f x+ \int_0^T \phi\langle\f  f , \vv + \f v \rangle \de t  \,.\label{eninboth}
\end{multline}
Reformulating~\eqref{eninv} by Lemma~\ref{lem:invar}, adding~\eqref{eninvv}, and subtracting~\eqref{eninboth}, let us deduce that 
\begin{multline}
-\int_0^T \phi'\frac{1}{2}\| \f v - \vv \|_{L^2(\Omega)}^2 \de t  +\nu \int_0^T \phi \| \nabla \f v - \nabla \vv \|^2_{L^2(\Omega)}\de t - \phi(0) \frac{1}{2}\| \f v -\vv \|_{L^2(\Omega)}^2 \\  \leq  \int_0^T\phi  \int_\Omega  \left ( ( \f v \cdot \nabla ) \f v \cdot \vv + ( \vv \cdot \nabla) \vv \cdot \f v\right )  \de \f x \de t + \int_0^T\phi \left  \langle \mathcal{A}(\vv),\vv- \f v\right \rangle  \de t\,\label{secrelen}
\end{multline}
for all $\phi\in\tilde{C}([0,T))$. 
In the following, we estimate the convective terms as in the proof of Serrin's result. 
Therefore, we use some standard manipulations using the skew-symmetry of the convective term in the last two arguments and the fact that $\vv$ is divergence free, to find
\begin{align*}
\int_\Omega  ( \f v \cdot \nabla ) \f v \cdot \vv + ( \vv \cdot \nabla) \vv \cdot \f v \de \f x ={}& \int_\Omega \big (  (( \f v-\vv) \cdot \nabla ) (\f v-\vv)  \cdot \vv + ( \vv \cdot \nabla) (\vv- \f v) \cdot (\f v- \vv) \big ) \de \f x\\ ={}&  \int_\Omega  (( \f v-\vv) \cdot \nabla ) (\f v-\vv)  \cdot \vv \de \f x\,.
\end{align*}
H\"older's, Gagliardo--Nirenberg's, and Young's inequality provide the estimate
\begin{align}
\begin{split}
 \int_\Omega  (( \f v-\vv) \cdot \nabla ) (\f v-\vv)  \cdot \vv \de \f x \leq{}&  \| \f v - \vv \|_{L^p(\Omega)} \| \nabla \f v - \nabla \vv \|_{L^2(\Omega)} \| \vv \|_{L^{2p/(p-2)}(\Omega)} \\ \leq{}& c_p \| \f v - \vv \|_{L^2(\Omega)} ^{(1-\alpha)}
 \| \nabla \f v - \nabla \vv \|_{L^2(\Omega)}^{(1+\alpha)}
  \| \vv \|_{L^{2p/(p-2)}(\Omega)} 
 \\ \leq{}& \frac{\nu}{2} \| \nabla \f v - \nabla \vv\|_{L^2(\Omega)}^2 + c \| \vv\|_{L^{2p/(p-2)}(\Omega)} ^{2/(1-\alpha)}
  \| \f v - \vv \|_{L^2(\Omega)}^2 \,,
\end{split}\label{Kest}
\end{align}
where $\alpha $ is chosen according to Gagliardo--Nirenberg's inequality by $$\alpha = d (p-2)/2p\quad\text{for}\quad  d\leq 2p/(p-2)\,.$$ 
Inserting this into~\eqref{secrelen} and 
replace $\phi$ by $ \varphi
e^{-\int_0^t\mathcal{K}(\vv) \de s }$ (or approximate it appropriately), we get 
\begin{multline*}
-\int_0^T \varphi '\frac{1}{2 }\| \f v(t) - \vv(t) \|_{L^2(\Omega)}^2 e^{-\int_0^t c \| \vv\|_{L^{2p/(p-2)}(\Omega)} ^{2/(1-\alpha)}
\de s }\de t - \frac{1}{2} \| \f v_0 - \vv(0 ) \|_{L^2(\Omega)}^2  \\ +\int_0^T \varphi  \left ( \frac{\nu}{2}\| \nabla \f v - \nabla \vv \|^2_{L^2(\Omega)}+\left \langle \mathcal{A}(\vv),\f v - \vv  \right \rangle \right )  e^{-\int_0^t c \| \vv\|_{L^{2p/(p-2)}(\Omega)} 
^{2/(1-\alpha)}
\de t }  \de s \leq 0 \,
\end{multline*}
for every smooth function $\vv\in\C^1([0,T];\mathbb{Y})$ and  all $\varphi \in \tilde{\C}([0,T])$.

 \end{proof}

 \begin{proposition}\label{prop:cont}
 Assume that there exists a weak solution $\f v\in \mathbb{X}$ with $\f v \in L^s(0,T;L^r(\Omega))$ with $2/s+d/r\leq 1$ to given right-hand side $\f f\in L^2(0,T;\f H^{-1} (\Omega))$ and initial datum $\f v_0\in \Ha$. For any $\f f^1\in L^2(0,T;\f H^{-1} (\Omega))$ and $\f v^1 _0\in \Ha$ the associated  dissipative solutions $ \f v^1$ fulfill the estimate
 \begin{align*}
 \mathcal{R}(\f v^1(t) | \f v(t)) +\frac{1}{2}\int_0^t  \mathcal{W}(\f v ^1 , \f v )  e^{\int_s^t\mathcal{K}(\f v) \de \tau} \de s \leq{}& \mathcal{R}(\f v^1_0 | \f v_0 ) e^{\int_0^t \mathcal{K}(\f v) \de s } \\&+ \frac{c}{\nu}\int_0^t \| \f f- \f f^1 \|_{\f H^{-1}(\Omega) }^2e^{\int_s^t\mathcal{K}(\f v) \de \tau}  \de s
 \end{align*}
 for a.e.~$t\in(0,T) $.
 \end{proposition}
\begin{remark}
The previous result gives no assertion on the continuous dependence on the data in general, but only conditionally, if a unique weak solution exists.  This can only be proven to be the case locally in time (see~\cite{temam}).
If such a strong solution does not exist, the continuous dependence only holds in the weak topology. For maximal dissipative solutions this continuous dependence is given in Theorem~\ref{thm:main}, below.
\end{remark} 
 
\begin{proof}[Proof of Proposition~\ref{prop:cont}]

Choosing $\vv$ to be the weak regular solution $\f v$ (or approximate it appropriately), we find by the Definition~\ref{def:diss} that
\begin{align}
\begin{split}
0 \geq{}&  \ \mathcal{R}(\f v^1(t)|\f v (t)) e^{-\int_0^t\mathcal{K}(\f v(s))\de s }  \de t -  \mathcal{R}(\f v^1 _0,\f v_0)  \\&{}+ \int_0^t  \left ( \mathcal{W}(\f v^1(s) | \f v (s)) +
\left \langle \mathcal{A}_{\f f^1}(\f v (s)) ,\f v^1 (s)-\f v(s)  \right \rangle \right )e^{-\int_0^s\mathcal{K}(\f v (s))\de s }   \de s\Big)\,,
\end{split}\label{continuous}
\end{align}
where $\mathcal{A}_{\f f^1}$ denotes the solution operator~\eqref{A} with $\f f$ replaced by $\f f^1$. 
Since $\f v$ is a solution for the right-hand side $\f f$, we may estimate
\begin{multline*}
\left \langle \mathcal{A}_{\f f^1}(\f v (t))- \mathcal{A}_{\f f}(\f v (t))   ,\f v^1 (t)-\f v(t)  \right \rangle = \left \langle \f f (t)- \f f^1(t) ,\f v^1 (t)-\f v(t)  \right \rangle \\ \geq - \frac{\nu}{4} \| \nabla \f v ^1 (t)-  \nabla \f v(t) \| _{L^2(\Omega)}^2 - \frac{c}{\nu} \| \f f(t) - \f f^1(t)\|_{H^{-1}(\Omega)} ^2 \,,
\end{multline*}
where Korn's inequality was applied (see~\cite[Thm.~10.15]{singular}). 
Reinserting this estimate into~\eqref{continuous}, applying Lemma~\ref{lem:invar} and multiplying by    
  $ e^{\int_0^t \mathcal{K}(\f v(\tau ) ) \de s }$, we find
the assertion of Proposition~\ref{prop:cont}.

\end{proof}

 \subsection{Well-posedness of maximal dissipative solutions}
 
We may state now the main theorem of this article:
\begin{theorem}\label{thm:main}
 Let $ \Omega \subset \R^d$ for $d\geq 2$ a Lipschitz domain, $\nu > 0$. Let $\mathcal{R}$, $\mathcal{W}$, $\mathcal{K}$, and $\mathcal{A}$ be given as above in~\eqref{def:nav}. Then there exists a unique maximal dissipative solution $\f v \in \mathbb{X}$ to every  $\f v_0\in \Ha$ and $\f f \in L^2 (0,T; H^{-1}(\Omega))$ 
  in the sense of Definition~\ref{def:maxdiss} and the maximal dissipative solution depends continuously on the initial datum and the right-hand side in the sense
$( \f v ^n_0 , \f f^n ) \ra ( \f v _0 , \f  f ) $ in $ \Ha \times L^2 (0,T; H^{-1}(\Omega))\cap L^1(0,T;L^2(\Omega))$, then  to every $n\in \N$ there exists a maximal dissipative solution $\f v^n \in \mathbb{X}$ and it holds 
$ \f v^n \stackrel{*}{\rightharpoonup} \f v $ in $\mathbb{X}$. 
 \end{theorem}
  
 
 \begin{remark}[Comparison to weak solutions]\label{weakvs}
 In the case that there exists a weak solution to the Na\-vier--Stokes equation complying to Serrin's uniqueness criterion, we observe that it is a maximal dissipative solution. 
 Indeed, let $\f v$ be a weak solution enjoying the regularity
 \begin{align*}
 \f v\in L^s(0,T;L^r(\Omega)) \quad \text{for } \frac{2}{s} + \frac{d}{r}\leq 1 \,,
 \end{align*}
 then the regularity measure $\mathcal{K}$ is bounded and we may use it as a test function $\vv$ (or rather approximate it by test functions) in the formulation of Definition~\ref{def:maxdiss}. Note that using density arguments, $\mathbb{Y}$ could be replaced by $ \mathbb{X}\cap L^s(0,T;L^r(\Omega))\cap W^{1,2}(0,T; ( \V  )^*)$ with $s$ and $r$ fulfilling the above condition.
 We observe that $\mathcal{F}(\f v | \f v) =0$, which is indeed the minimum since for every other function $\f u\in L^\infty(0,T;L^2_\sigma(\Omega)) \cap L^2(0,T;\V)$ emanating from the same initial datum, we observe that $\mathcal{F}(\f u|\f v) >0$. 
 Thus, maximal dissipative solutions coincide with weak solutions as long as the latter are unique.

 \end{remark}
 \begin{remark}[Reintroduction of the pressure]
 In this work, we only consider the velocity field for simplicity. Due to the fact that no equation is fulfilled  in the maximal dissipative solution concept, we do not 
have to worry about choosing the pressure in such a way that the full Navier--Stokes equation is fulfilled in a distributional sense (see~\cite{simon}). We propose to calculate the pressure by solving the usual elliptic Neumann boundary value problem
 \begin{align*}
- \Delta  p &= \nabla^2 :  \left ( \f v \otimes  \f v  \right ) - \di \f f \,, \quad &&\text{in } \Omega \\
\f n \cdot \nabla  p &=\f n \cdot \f f - \f n\cdot ( ( \f v \cdot \nabla ) \f v )\,, \quad &&\text{on }\partial \Omega 
 \end{align*}
 in a very weak sense with $ p \in L^1(\Omega)$ and  the additional normalization 
 $\int_\Omega p(t) \de \f x =0$ a.e.~in $(0,T)$.
 The previous formulation for the pressure especially makes sense, if one considers a suitable approximation of the Navier--Stokes equation, \textit{i.e.}, by a Galerkin approximation with a Galerkin space spanned by eigenfunctions of the Stokes operator. 
 Another approach would be to consider the Leray projection of the equations, see Remark~\ref{rem:proj}.
 \end{remark}
 \begin{remark}[other boundary conditions]
 In order to incorporate different boundary conditions it is sufficient to adapt the function space for the solution,\textit{ i.e.}, $\mathbb{X}$,  the test functions, \textit{i.e.},  $\mathbb{Y}$, and the formulation of the operator $\mathcal{A}$.  
 \end{remark}
 \begin{proof}[Proof of Theorem~\ref{thm:main}] 
The assumptions on the general problem in Assumption \textbf{(A1)} and \textbf{(A4)} are trivially fulfilled for the considered case, where $ \mathbb{V}= \Ha $, $\f K: \f H^2 \cap \V \ra \Ha^*$ is the Stokes operator, $\f L: \f H^2 \cap \V \ra \Ha^*$ the convective term and $p=2$. 
 It thus only remains to prove the existence of dissipative solutions in the case of the Navier--Stokes equations. 
In the case of $d=2,\,3$ or $4$, the existence of weak solutions to the Navier--Stokes equations is well known (see for instance~\cite{temam}). Due to Proposition~\ref{prop:diss}, this also proves the existence of dissipative solutions and thus maximal dissipative solutions.

For abitrary dimension, we may follow the standard discretization approach of Temam~\cite{temam} to prove the existence of dissipative solutions. 
Due to~\cite[p.~27f.]{temam} there exists a Galerkin basis consisting of  eigenfunctions $\f w_1, \, \f w_2 , \, \ldots \in \ \f H^1_{0,\sigma}  $ of the Stokes operator (with homogeneous Dirichlet boundary conditions). As is well known, the eigenfunctions form an orthogonal basis in $\Ha$ as well as in $\V$. Let $W_n= \spa \left \{ \f w_1, \dots , \f w _n\right \} $ ($n\in \N$)
and let $P_n : \Ha \longrightarrow W_n$ denote the
$\Ha$-orthogonal projection onto $W_n$.
The approximate problem is then given as follows: Find a solution $\f v^n \in \mathcal{AC}([0,T];V_n)$ solving the system
\begin{align}
\left ( \t \f v ^n + ( \f v ^n  \cdot \nabla) \f v ^n , \f w \right ) + \nu \left ( \nabla \f v^n \nabla \f w \right ) = \left \langle \f f , \f v^n \right \rangle \,, \quad \f v ^n(0) = P_n \f v_0 \,.\label{vdis}
\end{align}

A classical existence theorem (see Hale~\cite[Chapter I, Theorem 5.2]{hale}) provides, for every $n\in\N$, the existence of a maximal extended solution to the above approximate problem~\eqref{vdis} on an
interval~$[0,T_n)$ in the sense of Carath\'e{}odory.
 This theorem grants a solution  on $[0,T]$ if the solution undergoes no blow-up. With the standard \textit{a priori} estimates, we can exclude blow-ups and thus prove global-in-time existence. \label{sec:exloc}
Testing~\eqref{vdis} by $\f v^n$, we derive the standard energy estimates
\begin{align}
\frac{1}{2} \| \f v^n \|_{L^2(\Omega)}^2 + \nu\int_0^t  \| \nabla \f v ^n\|_{L^2(\Omega)}^2 \de s = \frac{1}{2}\| \f v_0 \| _{L^2(\Omega)}^2 + \int_0^t \langle \f f , \f v^n \rangle \de s \,\label{energydis}
\end{align}
for $\f f \in L^2(0,T;\f H^{-1}(\Omega)) \otimes L^1(0,T; \f L^2(\Omega)) $, there exists two functions $\f f_1\in L^2(0,T;\f H^{-1}(\Omega)) $ and $\f  f_2 \in L^1(0,T; \f L^2(\Omega))$ such that we may estimate with H\"older's, Young's, and Korn's inequality that
\begin{align*}
\langle \f f , \f v^n \rangle \leq \frac{\nu }{2}\|\nabla \f v^n \|_{L^2(\Omega)}^2 + \frac{C}{2\nu }\| \f f _1 \|_{\f H^{-1}(\Omega)} ^2 + \| \f f_2 \|_{L^2(\Omega)}  \left  (\| \f v^n \|_{L^2(\Omega)}^2+1\right ) \,.
\end{align*}
Inserting this into~\eqref{energydis} allows to apply a Version of Gronwall's Lemma in order to infer that $ \{ \f v^n \}$ is bounded and thus weakly compact in $\mathbb{X}$ such that there exists a $ \f v \in \mathbb{X} $ with
\begin{align*}
\f v^n \rightharpoonup \f v\quad  \text{in } \mathbb{X}\,.
\end{align*}

In order to show the convergence to dissipative solutions, we derive a discrete version of the relative energy inequality. Assume $\vv \in C^1([0,T];\mathbb Y)$. 
Adding~\eqref{energydis} and~\eqref{vdis} tested with $- P_n \vv$ (and integrated in time) and adding and subtracting the term $ \int_0^t\langle P_n \mathcal{A}(\vv) , \f v^n - \vv \rangle \de s $ leads to 
\begin{multline*}
\frac{1}{2} \| \f v - P_n\vv \|_{L^2(\Omega)}^2 \Big |_0^t+ \int_0^t \left ( \nu \| \nabla \f v ^n - \nabla P_n \vv \|_{L^2(\Omega)}^2 + \langle P_n \mathcal{A}(\vv) , \f v^n - \vv \rangle\right )  \de s \\
= \int_0^t\Big(  \left ( ( \f v^n \cdot \nabla ) \f v^n , P_n \vv \right ) + \left ( ( \vv \cdot \nabla) \vv , P_n\f v^n \right ) - \left ( ( \vv \cdot \nabla) \vv , P_n \vv \right ) \Big )  \de s \,. 
\end{multline*}
Note that order of the projection $P_n$ and the Stokes operator may be changed, due to the construction of the discrete spaces. 
By some algebraic transformations, we find
\begin{align}
\begin{split}
\left ( ( \f v^n \cdot \nabla ) \f v^n , P_n \vv \right ) +{}& \left ( ( \vv \cdot \nabla) \vv , P_n\f v^n \right ) - \left ( ( \vv \cdot \nabla) \vv , P_n \vv \right )  \\
={}&\left ( ( (\f v^n- P_n \vv) \cdot \nabla ) (\f v^n- P_n \vv)  , P_n \vv \right ) 
\\&+\left ( ( P_n \vv \cdot \nabla ) (\f v^n- P_n \vv)  , P_n \vv \right ) + \left ( ( \vv \cdot \nabla) \vv , \f v^n - P_n \vv \right ) 
\\={}&\left ( ( (\f v^n- P_n \vv) \cdot \nabla ) (\f v^n- P_n \vv)  , P_n \vv \right ) + \left ( ( \vv \cdot \nabla) \vv - ( P_n \vv \cdot \nabla) P_n \vv , \f v^n - P_n \vv \right ) \,.
\end{split}\label{discalc}
\end{align}

Similar to~\eqref{Kest}, we may estimate
\begin{align*}
\left ( ( (\f v^n- P_n \vv) \cdot \nabla ) (\f v^n- P_n \vv)  , P_n \vv \right ) \leq  \mathcal{W}(\f v^n | P_n \vv) + \mathcal{K}(P_n \vv) \frac{1}{2} \mathcal{R}( \f v | P_n\vv )\,
\end{align*}
For the second term on the right-hand side of~\eqref{discalc}, we observe
\begin{align*}
 \left ( ( \vv \cdot \nabla) \vv - ( P_n \vv \cdot \nabla) P_n \vv , \f v^n - P_n \vv \right ) ={}&  \left ( ( (\vv- P_n \vv) \cdot \nabla) \vv + ( P_n \vv \cdot \nabla) (\vv - P_n \vv) , \f v^n - P_n \vv \right )\,.
\end{align*}

In order to find the discrete version of the relative energy inequality, we apply the  Gronwall lemma,
\begin{multline*}
\mathcal{R}(\f v^n | P_n \vv ) e^{-\int_0^t \mathcal{K}(P_n \vv) \de \tau } + \int_0^t \left ( \mathcal{W}(\f v^n | P_n \vv ) + \langle  \mathcal{A}(P_n \vv ) , \f v^n -P_n \vv \rangle \right ) e^{-\int_s^t \mathcal{K}(P_n \vv)\de \tau } \de s 
\\ \leq \mathcal{R}(P_n \f v_0 | P_n \vv(0))  + \int_0^t \left ( ( (\vv- P_n \vv) \cdot \nabla) \vv  +P_n \vv \cdot \nabla) (\vv - P_n \vv),\f v^n - P_n \vv \right )  e^{-\int_s^t \mathcal{K}(P_n \vv)\de \tau } \de s  \,.
\end{multline*}
The strong convergence of the projection $P_n$, \textit{i.e}.,
\begin{align*}
\| P_n \vv - \vv \|_{L^\infty(0,T;\V)} \ra 0 \quad \text{as }n \ra \infty \quad \text{for all }\vv \in \C^1([0,T];\mathbb Y)\,,
\end{align*} allows to pass to the limit in the discrete relative energy inequality and attain the continuous one~\eqref{relen}. This proves the existence of dissipative solutions and thus, Theorem~\ref{thm:main}.

 \end{proof}
 \begin{remark}[Relative energy inequality for non-solenoidal test functions]\label{rem:proj} Note that the test function in the existence previous proof could also be chosen to have non-vanishing divergence, \textit{i.e}., $ \vv \in \C^1([0,T]; \f H^2 \cap \f H^1_0 )$. Passing to the limit in this formulation, with test functions that are not necessarily divergence-free, we would end up with a slightly different dissipative formulation. 
 
 The usual test function $\vv$ is always replaced by $P\vv$, where $ P$ denotes the Leray-projection onto the divergence-free functions. Furthermore, an additional  term would appear on the right-hand side:
 \begin{align*}
 \int_0^t \left ( ( (\vv- P \vv) \cdot \nabla) \vv+ (P \vv\cdot \nabla) (\vv-P\vv)  ,\f v - P \vv \right )  e^{-\int_s^t \mathcal{K}(P \vv)\de \tau } \de s\,.
 \end{align*}
Note that term $ \vv - P\vv$ only depends on the divergence of $\vv$ and vanishes with vanishing divergence.   This formulation may be more interesting from the numerical point of view, since there it can often only be guaranteed that the test functions are divergence free in the discrete sense and not in the continuous sense.

 \end{remark}
 \begin{remark}
 It is worth noticing that there is no stability property of the projection $P_n$ onto the discrete spaces needed. Usually $P_n$ has to be stable as a mapping on $\mathbb{Y}$ (for instance) in order to infer estimates on the time derivatives, which then gives by some version of the Aubin--Lions theorem strong convergence. In the end such strong convergence is needed to pass to the limit in the nonlinear terms. 
 Since no strong convergence is needed to pass to the limit in the dissipative formulation, the stability of the projection is not needed in our case. 
 \end{remark}

\section{ Euler equations\label{sec:eul}}
The general result of Theorem~\ref{thm:gen} is applied to the Euler equations and we discuss also the possibility of the measure-valued maximal dissipative solutions. 
\subsection{Well-posedness of maximal dissipative solutions}
A simple adaptation leads to the existence result for the Euler equations. 
First we recall the Euler equations for the sake of completeness. 
\begin{align*}
\t \f v + ( \f v \cdot \nabla ) \f v + \nabla p = \f f , \qquad \di \f v ={}& 0 \qquad \text{in }\Omega \times (0,T)\,,\\
\f v (0) ={}& \f v_0 \qquad \text{in } \Omega \,,\\
\f v \cdot \f n = {}& 0 \qquad \text{on }\partial \Omega\times (0,T) \,.
\end{align*}

For the Euler equations, the underlying spaces change to $\mathbb{X}:=L^\infty(0,T;L^2_{\sigma} (\Omega))$ for the solutions and $\mathbb{Y}:= \f H^1 \cap \Ha $ for the test functions. 
The definitions of the relative energy and the relative dissipation, as well as the solution operator are given as in~\eqref{def:nav} with $\nu=0$. The regularity measure changes to $\mathcal{K}(\vv)=\| (\nabla \vv)_{\sym,-}\|_{L^\infty(\Omega)}$, where $(\nabla \vv)_{\sym,-}$ denotes the negative part of the symmetrized gradient of $\vv$, \textit{i.e.}, $$ \| (\nabla \vv)_{\sym,-}\|_{L^\infty(\Omega)}= \left \| \left  ( \sup _{|\f a| = 1} - \left ( \f a ^T \cdot (\nabla \vv)_{\sym} \f a\right ) \right ) \right \|_{L^\infty(\Omega)}  \,.$$

We recall an existence result on dissipative solutions for the Euler equation by Pierre-Louis Lions~\cite[Sec~4.4]{lionsfluid}:
\begin{theorem}[Existence of dissipative solutions]\label{lions} 
 Let $ \Omega \subset \R^d$ for $d\geq 2$ a Lipschitz domain.
 Let $\mathcal{R}$, $\mathcal{W}$, and $\mathcal{A}$ be given as in~\eqref{def:nav} with $\nu=0$ such that $ \mathcal{W}\equiv 0$ and let $\mathcal{K}$ be given by $\mathcal{K}(\vv)=\| (\nabla \vv)_{\sym,-}\|_{L^\infty(\Omega)}$. Then there exists at least one function $\f v\in L^\infty(0,T;L^2_{\sigma}(\Omega))$ to every $\f v_0\in L^2_{\sigma} (\Omega)$ and $\f f \in L^2 (0,T; L^2(\Omega))$ such that
 \begin{align}
 \mathcal{R}(\f v (t) | \vv(t)) \leq \mathcal{R}(\f v_0, \vv(0)) e^{\int_0^t \mathcal{K}(\vv)\de s } + \int_0^t \left \langle\mathcal{A}(\vv), \vv-\f v\right \rangle     e^{\int_ s^t \mathcal{K}(\vv)\de \tau }\, \quad \text{for all } \vv \in \C^1 ([0,T]; \mathbb{Y})\,.\label{lionsdiss}
\end{align}

 \end{theorem}
\begin{remark}
Pierre Louis Lions also showed that $\f v$ enjoys the regularity $\f v \in \C_w([0,T];L^2(\Omega))$. We omit this here, since the regularity is not stable under the convergence with respect to $\mathbb{X}$. 
\end{remark} 
 We are now ready to state the existence  result for the Euler equations. 

\begin{theorem}\label{thm:euler}
%
Let the assumptions of Theorem~\ref{lions} be fulfilled. 
 Then there exists a unique maximal dissipative solution $\f v \in \mathbb{X}$ to every $\f v_0\in L^2_{\sigma} (\Omega)$ and $\f f \in L^2 (0,T; L^2(\Omega))$ 
  in the sense of Definition~\ref{def:maxdiss} and the maximal dissipative solution depends continuously on the initial datum and the right-hand side in the sense
$( \f v ^n_0 , \f f^n ) \ra ( \f v _0 , \f  f ) $ in $ \Ha \times L^1 (0,T; L^2 (\Omega))$, then  to every $n\in \N$ there exists a maximal dissipative solution $\f v^n \in \mathbb{X}$ and it holds 
$ \f v^n \stackrel{*}{\rightharpoonup} \f v $ in $\mathbb X$ . 
 \end{theorem}
 Since the existence of dissipative solutions is already know due to Theorem~\ref{lions}, the above theorem is a consequence of Theorem~\ref{thm:gen}.
 Additionally, we provide a conditional continuous dependence result similar to Proposition~\ref{prop:cont}. 
\begin{proposition}\label{prop:euler}
 Assume that there exists a unique weak solution $\f v\in \mathbb{X}$ to the Euler equations with $\f v \in L^1(0,T;W^{1,\infty}(\Omega))$ to given right-hand side $\f f\in L^1(0,T;\f L^2 (\Omega))$ and initial datum $\f v_0\in \Ha$. For any $\f f^1\in L^2(0,T;\f L^2 (\Omega))$ and $\f v^1 _0\in \Ha$ the associated  dissipative solution $ \f v^1$ fulfills the estimate 
 \begin{align*}
 \mathcal{R}(\f v^1(t) | \f v(t)) \leq \mathcal{R}(\f v^1_0 | \f v_0 ) e^{\int_0^t (\mathcal{K}(\f v)+1 )  \de s } + \frac{1}{2}\int_0^t \| \f f- \f f^1 \|_{\f L^2(\Omega)}^2e^{\int_s^t(\mathcal{K}(\f v)+1)  \de \tau}  \de s
 \end{align*}
 for a.e.~$t\in(0,T) $.
 \end{proposition} 
\begin{proof}[Proof of Proposition~\ref{prop:euler}]
As in the proof of Proposition~\ref{prop:cont}, we get~\eqref{continuous}. 
We continue by estimating 
\begin{multline*}
\left \langle \mathcal{A}_{\f f^1}(\f v (t))-\mathcal{A}_{\f f}(\f v (t)) ,\f v^1 (t)-\f v(t)  \right \rangle = \left \langle  \f f(t) - \f f^1 (t),\f v^1 (t)-\f v(t)  \right \rangle \\ \geq - \frac{1}{2} \|  \f v ^1 (t)-   \f v (t)\| _{L^2(\Omega)}^2 - \frac{1}{2} \| \f f (t)- \f f^1(t)\|_{L^2(\Omega)} ^2 \,. 
\end{multline*}
Inserting this into~\eqref{continuous} for the Euler equations and choosing $\phi=\varphi e^{-t}$, we find
\begin{align*}
0 \geq  \Big(& -\int_0^T \varphi'(t) \mathcal{R}(\f v^1(t)|\f v (t)) e^{-\int_0^t (\mathcal{K}(\f v(s))+1)\de s }  \de t - \varphi(0) \mathcal{R}(\f v^1 _0,\f v_0)  \\&{}+ \int_0^T \varphi(t) (
\left \langle \mathcal{A}_{\f f^1}(\f v (t)) ,\f v^1 (t)-\f v(t)  \right \rangle e^{-\int_0^t(\mathcal{K}(\f v (s))+1)\de s }   \de t\Big)\,.
\end{align*}
Applying Lemma~\ref{lem:invar} and multiplying by~$e^{\int_0^t\mathcal{K}(\f v (s)+1)\de s } $, implies the assertion. 
\end{proof}

\subsection{Measure valued formulation}
In this section, we want to define a measure-valued solution for the Euler equations and go a similar step into the uniqueness of solutions as in the case of dissipative solutions. 
\begin{definition}[measure-valued solution]\label{def:meas}
A function $\f v \in L^\infty(0,T;\f L^2_{\sigma}(\Omega)  )$ is called a measure-val\-ued solution to the Euler equations, if there exists a measure $ \f m \in  L^\infty (0,T;\mathcal{M}(\ov \Omega ; \mathbb{M}^{d\times d}_{\sym,+})$ such that
the equation is fulfilled in a measure-valued sense, \textit{i.e.}, 
\begin{align}
- \int_0^T \langle \f v , \partial_t \f \varphi \rangle + \left ( \f v \otimes \f v; \nabla \f \varphi \right )  + \langle \f m ; \nabla \f \varphi \rangle \de t + \left ( \f v _0 , \f \varphi(0)\right ) =  \int_0^T ( \f f , \f \varphi ) \de t \label{measeq}
\end{align}
for all $ \f \varphi \in \C_0^\infty([0,T)) \otimes \mathcal{V}$ and the energy inequality holds
\begin{align}
\frac{1}{2}\| \f v (t) \|^2_{\f L^2(\Omega)} + \frac{1}{2}\langle  \f m(t) ; I \rangle  \leq \frac{1}{2} \| \f v_0 \|_{\f L^2(\Omega)}^2 + \int_0^t \langle \f f , \f v \rangle \de s  \,.\label{measeneq}
\end{align}

\end{definition}

\begin{remark}
Measure valued solutions to the Euler equations are known to enjoy the weak strong uniqueness property (see~\cite{weakstrongeuler}). 
If $\f m \equiv 0 $, they fulfill the equation in the weak sense. 
Note that 
the dual pairing of the measure and a continuous function is defined via 
$ \langle \f m , \f A \rangle := \langle \f m , \f A \rangle_{\mathcal{M}(\ov \Omega ; \mathbb{M}^{d\times d}), \C(\ov \Omega ; \mathbb{M}^{d\times d})} = \int_{\ov \Omega } \f A : \de \f m $.
\end{remark}
\begin{remark}
This formulation differs slightly form the usual formulation by Di-Perna--Majda (see~\cite{dipernamajda} or~\cite{weakstrongeuler}). 
Here, we just add and subtract the term $\left ( \f v \otimes \f v; \nabla \f \varphi \right ) $, which results in a redefinition of the measure $\f m$ (compare~\cite{maxbreit}). 

\end{remark}
\begin{definition}[maximal dissipative measure-valued solutions]\label{def:maxmeas}
A function $\f u$ is called a maximal dissipative solution, if $ \f v \in \mathbb{X}$   is the solution of the following optimization problem
\begin{align*}
\min_{\f v \in \mathbb{X}} \frac{1}{2} \int_0^T\|  \f v ( t)\|_{L^2(\Omega)}^2  \de t \quad \text{such that $\f v $ is a measure-valued solution according to Definition~\ref{def:diss}. }
\end{align*}
\end{definition}

\begin{theorem}
 Let $\f v_0\in L^2_{\sigma} (\Omega)$, and $\f f \in L^1 (0,T; L^2(\Omega))$ be given.  Then there exists a measure-valued solution in the sense of Definition~\ref{def:meas}. The solution set is convex and closed such that there exists a unique maximal dissipative solution according to Definition~\ref{def:maxmeas}. 
\end{theorem}
\begin{proof}
The proof is divided in several steps: \textit{Existence of solutions.}
The existence of measure-valued solution was already proven in~\cite{dipernamajda} (see also~\cite{weakstrongeuler}). 
The existence can be proven by the usual vanishing viscosity approach, where the measure valued solution to the Euler equations is the limit of the weak solutions to the Navier--Stokes equations. 
For a suitable approximation $\f v^n\ra \f v$ in $L^\infty(0,T;\Ha)$, the defect measure is defined as 
\begin{align*}
\int_0^T \langle \f m ; \f A   \rangle \de t = \lim _{n\ra \infty} \int_0^T \int_\Omega \left ( \f v^n \otimes \f v^n - \f v \otimes \f v \right ) : \f A  \de \f x \de t 
\end{align*}
 for all $ \f A \in \C_c^\infty([0,T)) \otimes \C (\ov  \Omega ; \mathbb M^{d\times d})$. The measure $\f m $ takes values in the  symmetric matrices since it is the limit of symmetric matrices. 
Due to the weakly lower semi-continuity of convex functionals~\cite{ioffe}, $\f m$ takes also values in the set of positive definite matrices, \textit{i.e.,}
\begin{align*}
\int_0^T \langle \f m ; \f a \otimes \f a   \rangle \de t ={}&\lim _{n\ra \infty} \int_0^T \int_\Omega \left ( \f v^n \otimes \f v^n - \f v \otimes \f v \right ) : \f a \otimes \f a   \de \f x \de t \\\geq {}& \liminf _{n\ra \infty} \int_0^T \int_\Omega \left ( \f v^n \cdot \f a \right )^2  -\left ( \f v \cdot  \f a\right )^2   \de \f x \de t  \geq 0
\end{align*}
for all $\f a \in \C_c^\infty([0,T)) \otimes \C(\ov\Omega ; \R^d ) $. 

\textit{Solution set is convex.}
 Let $\f v_1$ and $\f v_2$ be two measure-valued solutions according to Definition~\ref{def:meas} with the measures $\f m_1$ and $\f m_2$ respectively. 
 A simple calculation shows
 \begin{multline*}
 ( \lambda \f v_1 + (1-\lambda) \f v_2) \otimes  ( \lambda \f v_1 + (1-\lambda) \f v_2) =\\ \lambda \left ( \f v_1 \otimes \f v_1 \right ) + ( 1-\lambda ) \left ( \f v _2 \otimes \f v_2\right ) - \lambda (1- \lambda ) \left ( ( \f v_1 - \f v _2) \otimes ( \f v _1 - \f v_2 ) \right  )\,.
 \end{multline*}
This implies that $\f v:=\lambda \f v _1 + (1-\lambda) \f v_2 $ fulfills the equation~\eqref{measeq} with the measure $$\f m:= \lambda \f m_1 + (1-\lambda ) \f m_2 + \lambda (1-\lambda)  \left ( ( \f v_1 - \f v _2) \otimes ( \f v _1 - \f v_2 ) \right  )\,,$$
which is again a positive definite matrix point wise a.e. in $\Omega \times (0,T)$. 
Similar observations imply that the energy inequality~\eqref{measeneq} also holds for $\f v$ and $\f m$. 

\textit{Solution set is closed.}
Let $\{\f v^n\}\subset L^\infty(0,T;\Ha)$ be a sequence of measure-valued solutions according to Definition~\ref{def:meas} with associated measures $ \{\f m^n\} \subset  L^\infty (0,T;\mathcal{M}(\ov \Omega ; \mathbb{M}^{d\times d}_{\sym,+}))$. Since the energy inequality~\eqref{measeneq} is fulfilled for every $n\in \N$, we may deduce that $ \{ \f v^n \} $ is bounded in $ L^\infty(0,T;\Ha)$ and $\{\f m^n\} $ is bounded in $  L^\infty (0,T;\mathcal{M}(\ov \Omega ; \mathbb{M}_{\sym,+}^{d\times d} ))$ independent of $n$. 
Note that we may estimate the right-hand side by $ \langle \f f , \f v \rangle \leq \| \f f \|_{L^2(\Omega)} \| \f v \| _{L^2(\Omega)} \leq \| \f f \|_{L^2(\Omega)} (1+\| \f v \| _{L^2(\Omega)} ^2)$ such that the boundedness follows by Gronwall's Lemma. 
This allows to infer the existence of a subsequence such that
\begin{align*}
\f v^n \stackrel{*}{\rightharpoonup} \f v  \quad \text{in } L^\infty(0,T;\Ha) \qquad \text{and} \qquad \f m^n \stackrel{*}{\rightharpoonup} \tilde{\f m} \quad \text{in }L^\infty (0,T;\mathcal{M}(\ov \Omega ; \mathbb{M}_{\sym,+}^{d\times d} ))\,.
\end{align*}
Since $\{ ( \f v^n \otimes \f v^n ) \} $ is also bounded in $L^\infty (0,T;\mathcal{M}(\ov \Omega ; \mathbb{M}_{\sym,+}^{d\times d} ))$, we may select another subsequence such that there exists a measure $ \bar{\f m} \in L^\infty (0,T;\mathcal{M}(\ov \Omega ; \mathbb{M}_{\sym,+}^{d\times d} ))$ with 
\begin{align*}
\int_0^T \langle \bar{\f m }; \f A   \rangle \de t = \lim _{n\ra \infty} \int_0^T \int_\Omega \left ( \f v^n \otimes \f v^n - \f v \otimes \f v \right ) : \f A  \de \f x \de t \,
\end{align*}
 for all $ \f A \in \C_c^\infty([0,T)) \otimes \C (\ov  \Omega ; \mathbb M^{d\times d})$. Defining $ \f m : = \bar{\f m} + \tilde{\f m} \in  L^\infty (0,T;\mathcal{M}(\ov \Omega ; \mathbb{M}_{\sym,+}^{d\times d} ))$, we observe that in the limit, the measure-valued formulation~\eqref{measeq} is fulfilled by $\f v $ and $\f m$. Similar, this can be observed for the energy inequality, where we may use Lemma~\ref{lem:invar} to pass to the limit in the energy inequality as an in-between-step. 

\textit{Well-posedness of maximal dissipative solutions.} 
 Since the solution set is closed and convex, the optimization problem of Definition~\ref{def:maxmeas} is well-defined and there exists a unique minimizer of the functional since it is strictly convex. 

\end{proof}

\begin{remark}
It would also be possible to prove continuous dependence in this case, at least in the weak$^*$ topology again with an adapted measure. 
Hence, it would also be possible to show the well-posedness of maximal dissipative measure-valued solutions in the case of the Navier--Stokes equations.

We refrained from defining maximal dissipative solution in the proposed measure valued sense for several reasons. First of all, the measure $\f m$ has way more degrees of freedom and seems to be some auxiliary variable, which is only needed to fulfill the equation in some very weak sense. But as explained above, the equation may not be the right relation to approximate. 
Secondly, the measure $\f m$ in the definition of measure-valued solutions does not appear in the minimization process of Definition~\ref{def:maxmeas} of maximal dissipative measure valued solutions. Thus, the measure may becomes arbitrarily large in this formulation.  On the other hand, if the measure $\f m$ is introduced into the objective function in Definition~\ref{def:maxmeas}, the solutions may not be unique since the associated admissible set is not convex anymore (there are also other selection criteria proposed, see~\cite{maxbreit}).
Thirdly, for more involved coupled systems, it was observed that natural numerical approximation rather converge to dissipative solutions in the sense of Definition~\ref{def:diss} instead of measure-valued solutions in the sense of Definition~\ref{def:meas} (see~\cite{approx} and~\cite{nematicelectro}). 
Dissipative solution also fit better into a general framework as in Section~\ref{sec:gen}

%
 
\end{remark}


\begin{thebibliography}{10}

\bibitem{raymond}
D.~Ars\'{e}nio and L.~Saint-Raymond.
\newblock {\em From the {V}lasov-{M}axwell-{B}oltzmann system to incompressible
  viscous electro-magneto-hydrodynamics. {V}ol. 1}.
\newblock EMS Monographs in Mathematics. European Mathematical Society (EMS),
  Z\"{u}rich, 2019.
  
  \bibitem{nematicelectro}
{\'L}.~Ba{\v{n}}as, R.~Lasarzik, and A.~Prohl.
\newblock {Numerical analysis for nematic electrolytes}.
\newblock {\em \textsc{WIAS} Preprint, No. 2717, Berlin}, 2020.

\bibitem{maxbreit}
D.~Breit, E.~Feireisl, and M.~Hofmanov{\'a}.
\newblock Dissipative solutions and semiflow selection for the complete euler
  system.
\newblock  {\em Comm. Math. Phys.}, 376(2):1471--1497,
  2020.

\bibitem{weakstrongeuler}
{ Y.~Brenier, C.~De~Lellis \& L.~Sz{\'e}kelyhidi, Jr.}
\newblock Weak-strong uniqueness for measure-valued solutions.
\newblock {\em Comm. Math. Phys.}, 305(2):351--361, 2011.


\bibitem{buckmaster}
T.~Buckmaster and V.~Vicol.
\newblock Nonuniqueness of weak solutions to the {N}avier--{S}tokes equation.
\newblock {\em Ann. Math.}, 189(1):101--144, 2019.

\bibitem{maximaldiss}
C.~Dafermos.
\newblock Maximal dissipation in equations of evolution.
\newblock  {\em J.~Differ.~Equ.}, 252(1):567 -- 587, 2012.


\bibitem{dafermosscalar}
C.~M. Dafermos.
\newblock The entropy rate admissibility criterion for solutions of hyperbolic
  conservation laws.
\newblock {\em J.~Differ.~Equ.}, 14(2):202 -- 212, 1973.

\bibitem{dipernamajda}
{ R.~J. DiPerna \& A.~J. Majda}.
\newblock Oscillations and concentrations in weak solutions of the
  incompressible fluid equations.
\newblock {\em Comm. Math. Phys.}, 108(4):667--689, 1987.


\bibitem{ekeland}
I.~Ekeland and R.~Temam.
\newblock {\em Convex Analysis and Variational Problems}.
\newblock Classics in Applied Mathematics. SIAM, Philadelphia, 1999.

\bibitem{sabine}
E.~Emmrich, S.~H. Klapp, and R.~Lasarzik.
\newblock Nonstationary models for liquid crystals: A fresh mathematical
  perspective.
\newblock {\em J. Non-Newton. Fluid Mech.}, 259:32--47, 2018.


\bibitem{evans}
L.~C. Evans.
\newblock {\em Partial differential equations}, volume~19 of {\em Graduate
  Studies in Mathematics}.
\newblock American Mathematical Society, Providence, RI, second edition, 2010.

\bibitem{mill}
{ C.~L. Fefferman}.
\newblock Existence and smoothness of the {N}avier--{S}tokes equation.
\newblock In: {\em The millennium prize problems}, S. 57--67. Clay Math. Inst.,
  Cambridge, MA, 2006.
  
  \bibitem{singular}
{ E.~Feireisl \& A.~Novotn{\'y}}.
\newblock {\em Singular limits in thermodynamics of viscous fluids}.
\newblock Birkh{\"a}user, Basel, 2009.





\bibitem{gil}
D.~Gilbarg and N.~S. Trudinger.
\newblock {\em Elliptic partial differential equations of second order}.
\newblock Springer-Verlag, Berlin, 2001.



\bibitem{generic}
M.~{Grmela} and H.~C. {{\"O}ttinger}.
\newblock {Dynamics and thermodynamics of complex fluids. I. Development of a
  general formalism}.
\newblock {\em Phys.~Rev.~E}, 56(6):6620--6632, Dec 1997.

\bibitem{hale}
{ J.~Hale}.
\newblock {\em Ordinary differential equations}.
\newblock Wiley-Interscience, New York, 1969.

\bibitem{ioffe}
A.~Ioffe.
\newblock On lower semicontinuity of integral functionals. {I}.
\newblock {\em SIAM J. Control Optim.}, 15(4):521--538, 1977.


\bibitem{Isett}
P.~Isett.
\newblock A proof of {O}nsager's conjecture.
\newblock {\em Ann.~Math.}, 188(3):871--963, 2018.


\bibitem{diss}
R.~Lasarzik.
\newblock Dissipative solution to the {E}ricksen--{L}eslie system equipped with the
  Oseen--Frank energy.
\newblock {\em Z. Angew. Math. Phy.}, 70(1):8, 2018.

\bibitem{masswertig}
R.~Lasarzik.
\newblock Measure-valued solutions to the {E}ricksen--{L}eslie model equipped
  with the {O}seen--{F}rank energy.
\newblock {\em Nonlin. Anal.}, 179:146--183, 2019.

\bibitem{weakstrong}
R.~Lasarzik.
\newblock Weak-strong uniqueness for measure-valued solutions to the
  {E}ricksen--{L}eslie model equipped with the Oseen--Frank free energy.
\newblock {\em J. Math. Anal. Appl.}, 470(1):36--90, 2019.

\bibitem{approx}
R.~Lasarzik.
\newblock Approximation and optimal control of dissipative solutions to the
  {E}ricksen--{L}eslie system.
\newblock {\em Numer.~Func.~Anal.~Opt.},
  40(15):1721--1767, 2019.
  
  \bibitem{nsch}
R.~Lasarzik,.
Analysis of a thermodynamically consistent Navier--Stokes--Cahn--Hilliard model.
\newblock {\em WIAS-Preprint 2739}, 2020.
  
\bibitem{lrs}
R.~Lasarzik, E.~Rocca, G.~Schimperna.
Weak solutions and weak-strong uniqueness to a thermodynamically consistent phase-field model.
\newblock {\em WIAS-Preprint 2608}, 2019.

  
  \bibitem{leray}
J.~Leray.
\newblock Sur le mouvement d'un liquide visqueux emplissant l'espace.
\newblock {\em Acta Mathematica}, 63(1):193--248, 1934.

\bibitem{LionsBoltzman}
P.-L. Lions.
\newblock Compactness in {B}oltzmann's equation via {F}ourier integral
  operators and applications. {I}, {II}.
\newblock {\em J. Math. Kyoto Univ.}, 34(2):391--427, 429--461, 1994.

\bibitem{lionsfluid}
P.-L. Lions.
\newblock {\em Mathematical topics in fluid mechanics. {V}ol. 1}.
\newblock The Clarendon Press, New York, 1996.

\bibitem{mielke}
A.~Mielke.
\newblock On evolutionary {$\varGamma$}-convergence for gradient systems.
\newblock In A.~Muntean, J.~Rademacher, and A.~Zagaris, editors, {\em
  Macroscopic and large scale phenomena: coarse graining, mean field limits and
  ergodicity}, volume~3 of {\em Lect. Notes Appl. Math. Mech.}, pages 187--249.
  Springer, 2016.


\bibitem{RoubicekMeasure}
T.~Roub{\'{\i}}{\v{c}}ek.
\newblock {\em Relaxation in optimization theory and variational calculus},
  volume~4 of {\em de Gruyter Series in Nonlinear Analysis and Applications}.
\newblock Walter de Gruyter \& Co., Berlin, 1997.

\bibitem{serrin}
J.~Serrin.
\newblock On the interior regularity of weak solutions of the
  {N}avier--{S}tokes equations.
\newblock {\em Arch. Rational Mech. Anal.}, 9:187--195, 1962.

\bibitem{simon}
J.~Simon.
\newblock On the existence of the pressure for solutions of the variational
  {N}avier--{S}tokes equations.
\newblock {\em J.~Math.~Fluid~Mech.}, 1(3):225--234, 1999.

\bibitem{temam}
R.~Temam.
\newblock {\em The Navier-Stokes equations: Theory and numerical analysis}.
\newblock American Math. Soc., New York, 1984 (corrected reprint 2001).


\bibitem{viscoelsticdiff}
D.~A. Vorotnikov.
\newblock Dissipative solutions for equations of viscoelastic diffusion in
  polymers.
\newblock {\em J. Math. Anal. Appl.}, 339(2):876 -- 888, 2008.



\end{thebibliography}
\end{document}